\documentclass[11pt]{amsart}
\usepackage{amsmath,mathtools}
\usepackage{amsmath, amsthm, amssymb, amsfonts, enumerate}

\usepackage[colorlinks=true,linkcolor=blue,urlcolor=blue]{hyperref}
\usepackage{dsfont}
\usepackage{color}
\usepackage{geometry}
\usepackage{todonotes}
\usepackage{epstopdf}
\usepackage{tikz}
\usetikzlibrary{patterns}
\usetikzlibrary{patterns.meta}
\usetikzlibrary{arrows.meta}
\usepackage{pgfplots}
\usepackage{caption}
\usepgfplotslibrary{fillbetween}
\newcommand{\stkout}[1]{\ifmmode\text{\sout{\ensuremath{#1}}}\else\sout{#1}\fi}

\newtheorem{theorem}{Theorem}[section]
\newtheorem{remark}[theorem]{Remark}
\newtheorem{assumption}[theorem]{Assumption}
\newtheorem{lemma}[theorem]{Lemma}
\newtheorem{proposition}[theorem]{Proposition}

\newtheorem{definition}[theorem]{Definition}

\def \E{\mathsf{E}}

\def \P{\mathsf{P}}
\def \R{\mathbb{R}}
\def \F{\mathbb{F}}

\def\d{\mathrm{d}}

\def\D{\mathrm{D}}

\def\ss{\mathrm{\sigma}}

\definecolor{red}{rgb}{1.0,0.0,0.0}

\definecolor{blu}{rgb}{0.0,0.0,1.0}

\definecolor{gre}{rgb}{0.03,0.50,0.03}

\renewcommand{\tilde}{\widetilde}

\title[Singular Control of Diffusions and their Running Infimum or Supremum]{On the Singular Control of a Diffusion and its Running Infimum or Supremum} 

\author[Ferrari]{Giorgio Ferrari}
\author[Rodosthenous]{Neofytos Rodosthenous}

\address{G.~Ferrari: Center for Mathematical Economics (IMW), Bielefeld University, Universit\"atsstrasse 25, 33615, Bielefeld, Germany}
\email{\href{mailto:giorgio.ferrari@uni-bielefeld.de}{giorgio.ferrari@uni-bielefeld.de}}
\address{N.~Rodosthenous: Department of Mathematics, University College London, 25 Gordon Street, London WC1H 0AY, United Kingdom}
\email{\href{mailto:n.rodosthenous@ucl.ac.uk}{n.rodosthenous@ucl.ac.uk}}

\date{\today}

\numberwithin{equation}{section}

\begin{document}

\begin{abstract} 
We study a class of singular stochastic control problems for a one-dimensional diffusion $X$ in which the performance criterion to be optimised depends explicitly on the running infimum $I$ (or supremum $S$) of the controlled process. 
We introduce two novel integral operators that are consistent with the Hamilton-Jacobi-Bellman equation for the resulting two-dimensional singular control problems. 
The first operator involves integrals where the integrator is the control process of the two-dimensional process $(X,I)$ or $(X,S)$; 
the second operator concerns integrals where the integrator is the running infimum or supremum process itself. 
Using these definitions, we prove a general verification theorem for problems involving two-dimensional state-dependent running costs, costs of controlling the process, costs of increasing the running infimum (or supremum) and exit times. 
Finally, we apply our results to explicitly solve an optimal dividend problem in which the manager's time-preferences depend on the company's historical worst performance. 
\end{abstract}

\maketitle

\smallskip


\smallskip

{\textbf{MSC2020 subject classification}}: 93E20; 60J60; 49L12; 91B70.


\section{Introduction}
\label{introduction}

This paper deals with a class of singular stochastic control problems in which the decision maker can control via a nondecreasing process $\D$ the trajectory of a one-dimensional It\^o-diffusion $X^{\D}$ and of its associated running infimum $I^{\D}$. The aim is to maximise a discounted expected reward functional in which the time-horizon is given by the first absorption time of $X^{\D}$ at a given threshold, dynamically depending on the infimum process. 
In its general formulation, the expected reward functional consists of the sum of three integrals. 
The first one is the classical integral with respect to the Lebesgue measure of an instantaneous profit function depending on both $X^{\D}$ and $I^{\D}$. 
The other two integrals involve functions of the current levels $(X^{\D}_t,I^{\D}_t)$, which are integrated with respect to the (random) Borel-measures induced by the monotone processes $\D$ and $I^{\D}$. 

The first  main novelty of the present work consists in the definition of two novel integral operators which are consistent with the Hamilton-Jacobi-Bellman equation for singular stochastic control problems with state-process $(X^{\D}_t,I^{\D}_t)$. The challenge is due to the fact that jumps of the process $\D$ can induce jumps of $X^{\D}$ and $I^{\D}$ as well, so that understanding how to properly integrate over those jump times is a nontrivial task. 
As a matter of fact, depending on the size of such jumps of $\D$, there might be scenarios in which: 
\begin{enumerate}
\item[(a)] only the process $X^{\D}$ has a discontinuity, 
\item[(b)] both $X^{\D}$ and $I^{\D}$ have discontinuities with the same jump sizes,
\item[(c)] both $X^{\D}$ and $I^{\D}$ have discontinuities, but the sizes of their jumps are different.
\end{enumerate}
To the best of our knowledge, singular stochastic control problems with the aforementioned features are addressed in this paper for the first time.
For completeness, we also provide the corresponding framework for singular stochastic control problems involving a controlled one-dimensional It\^o-diffusion $X^{\D}$ and its associated controlled running supremum process $S^{\D}$.

On the one hand, this paper extends the literature on singular stochastic control problems in which the only state variable is a singularly controlled It\^o process $X^{\D}$, and the marginal cost of control is state-dependent.
A first contribution in this direction (in the sense of scenario (a) above) has been provided in \cite{Zhu}, whose definition of integral has been later extensively used in the singular control literature. 
In particular, these have been used for one-sided singular controls of classical diffusions \cite{AlMotairiZervos} and reflected diffusions \cite{Ferrari},
for combinations of one-sided singular controls and stopping \cite{LonZervos}, and for non-zero-sum games \cite{Dammannetal}, among others, 
with various applications in economics and finance, 
and especially in models with multiplicative price impact \cite{Belarevetal1, 
Belarevetal2, 
DammannFerrari,  
GZ15, 
HMP19}. 
The new definition in this paper essentially addresses the fact that the control process $\D$ can induce all three scenarios (a)--(c) in the same control problem, depending on the location of the process $(X^\D,I^\D)$ in the state-space, and also takes care of a resulting peculiar `hockey-stick' movement of the controlled process $(X^\D,I^\D)$ in the scenario (c); 
see Section \ref{sec:processsetting} for details and Figure \ref{Fig2} for examples of such jumps. 

On the other hand, the definition of the integral with respect to the (random) Borel-measure induced by the {\it controlled} running infimum process $I$ (or equivalently the {\it controlled} running supremum $S$) is completely novel in the literature. 
Such processes $I$ and $S$ and, possibly, integrals with respect to them, appeared in the literature on optimal stopping problems (see, e.g., \cite{GapRod14, GapRod16a, GapRod16b, GuoZerv10, KyprOtt14, Ott13,  Peskir98, RodZerv17, SheppShir93}, amongst many others). However, in those papers, $I$ and $S$ have always been considered to be continuous and uncontrolled, which results into the consideration of integrals only of the classical Lebesgue-Stieltjes type.

The second main theoretical contribution of this paper is that the two novel integral operators are then exploited in a verification theorem, proposing testable sufficient conditions under which a smooth solution to a variational inequality with gradient constraint identifies with the value function of the considered control problem. In particular, if there exists an (admissible) control process $D^{\star}$ that keeps with minimal effort the underlying state process $(X^{\D^{\star}},I^{\D^{\star}})$ in the region where the gradient constraint is not binding, then such $D^{\star}$ is optimal. 

As an application of the previous theoretical results, we then consider an optimal dividend problem in which the historical worst performance of a company's surplus process impacts the time-preference of the manager for paying dividends. 
This aims at maximising the expected discounted dividends to be paid by controlling the dividend process $\D$, under the constraint that the company defaults at the first time the controlled surplus process $X^{\D}$ (here a drifted Brownian motion, as it is customary in the literature; see, e.g., \cite{AsmussenTaksar, JeanblancShyriaev, LokkaZervos, RadnerShepp}) goes to zero. 
We model the evolution of the company's historical worst performance via the running infimum process $I^{\D}:=i \wedge \inf_{0\leq s \leq t}X^{\D}_s$. 
As a matter of fact, a company with a relatively high value $I^{\D}$ has been performing relatively well in the past, and we assume that this raises the manager's (and investors') confidence, who 
thus prefers to pay dividends sooner than later. 
On the contrary, the manager prefers to postpone dividend payments when the company's worst performance decreases and the company is thus under financial strain (such a behaviour has been observed -- and actually recommended by the ECB -- during the Covid-19 pandemic, and it is in line with the findings of \cite{DenisOsobov, FamaFrench}, among others, that observe that the propensity of paying dividends is lower in less profitable companies). 
We capture this by assuming that the manager's time-preference is described by a discount factor of the form $\exp\{-\rho t - q I^{\D}_t\}$, for some $\rho>0$ and $q\geq0$. 
Those two parameters represent the ``traditional" constant time-preference rate of the manager and the sensitivity of the manager with respect to the company's performance,~respectively.  

We prove that it is optimal to pay dividends as soon as the surplus process is larger or equal to a running infimum-dependent threshold $b$. 
This is given as the unique solution to a nonlinear first-order ordinary differential equation (ODE), whose initial datum coincides with the free boundary obtained in the classical De Finetti problem when $q=0$, studied in  \cite{LokkaZervos}, among many others. 
In particular, the optimal dividend policy is such that the optimally controlled surplus process $X^{\D^{\star}}$ is reflected downwards on $b(I^{\D^{\star}})$ at any strictly positive time. 
At initial time, instead, a lump-sum payment is optimal whenever $x>b(i)$. 
Interestingly, the amplitude of such a jump depends on the position of the initial worst performance $i$ with respect to a critical level $i^{\star}$, such that 
(i) if $i\leq i^\star$, then the lump-sum payment is equal to $x-b(i)$ (thus changing only the first coordinate of the process $(X^\D,I^\D)$ as in the top right panel in Figure \ref{Fig2}), while 
(ii) if $i > i^\star$, then the lump-sum payment is equal to $x-b(i^\star)$ (thus changing both coordinates of the process $(X^\D,I^\D)$ either equally as in bottom left panel or unequally as in bottom right panel in Figure \ref{Fig2}; see Section \ref{sec:v} below for details).
In this respect, the optimal dividend policy $\D^{\star}$ is substantially different to that of the classical De Finetti's problem, which is instead characterised by the reflection at a constant threshold. 
Furthermore, we prove that this boundary function is decreasing in the company's worst performance. 
This implies that for a relatively well-performing company with large $i$, the manager (with higher confidence in the company) is willing to pay dividends sooner than later, while the manager of a company with deteriorating worst performance should optimally postpone dividend payments until the surplus increases to a higher value.
However, interestingly, $\D^{\star}$ and the associated (optimal) value converge to those of the classical De Finetti's problem as $q\downarrow 0$.

The analysis and complete solution to the aforementioned dividend problem with multiplicative time-preference impact is the third main contribution of this paper.
This contributes to the large literature on optimal dividend problems and extensions of the classical De Finetti's problem. Without having the aim of providing here an exhaustive literature review on the topic, we only mention the earlier contribution by \cite{JeanblancShyriaev}, who studies the one-dimensional formulation of the dividend problem with bounded-velocity and singular controls; the works \cite{DeAngelisEkstrom},  \cite{FerrariSchuhmann}, and \cite{LokkaZervos}, who consider problems of optimal dividend with absorption and/or capital injections for a one-dimensional singularly controlled drifted Brownian motion, possibly over a finite time-horizon; the works \cite{Cadenillas}, \cite{Ferrari-etal}, and \cite{Pistorius}, where optimal dividend problems with regime-switching are studied; \cite{Bandini-etal} where a problem of optimal dividend with stochastic discounting force is solved.

The characterisation of the optimal policy via a free-boundary function solving an ODE is a well-established feature in the optimal stopping literature, motivated by applications in mathematical finance, economics, and quickest detection. 
It has been extensively documented in problems involving the extremum processes of diffusions; see, e.g.~\cite{DGM24, GHP13, GuoZerv10, Peskir98, RodZerv17, SheppShir93}, as well as of diffusion-type processes \cite{GapRod14, GapRod16a} and of L\'evy processes; see, e.g.~\cite{KyprOtt14, Ott13, RZ20}, among others.
Related two-dimensional degenerate singular control problems are studied in \cite{DeAngelisEkstromOlofsson, Ekstrom-etal}, where the optimal policy is likewise characterised by a free-boundary solving an ODE. 
While this feature is shared with the dividend problem studied in Section \ref{sec:optdiv} of the present paper, the structure of the state process and the nature of the control problem differ substantially: 
here the state consists of a singularly controlled It\^o diffusion and its controlled running infimum process, whereas in those works the controlled process is purely degenerate and the second state variable is either uncontrolled \cite{DeAngelisEkstromOlofsson} or only indirectly affected (via its drift) by the control \cite{Ekstrom-etal}. 
These differences lead to a markedly different analytical framework.

The rest of the paper is organised as follows. Section \ref{sec:problemsetting} gives the probabilistic setup. 
Section \ref{sec:problem} introduces the considered general formulation of the control problem for two-dimensional settings where the dynamics are driven by a controlled one-dimensional diffusion and its running extremum process. It also introduces the novel definitions of integral operators. 
Section \ref{sec:verthm} then presents the Hamilton-Jacobi-Bellmann equation and proves a verification theorem in the general setting of Section \ref{sec:problem}. 
The problem of optimal dividends with multiplicative time-preference impact is then introduced and solved in Section \ref{sec:optdiv}. Finally, Section \ref{sec:sup} shows how to incorporate also the case of a dependence with respect to the running supremum (rather than the running infimum) in the performance criterion to be optimised.


\section{Setup for a state-space involving the running infimum process}
\label{sec:problemsetting}

Let $(\Omega,\mathcal{F},\P)$ be a complete probability space, $B$ a one-dimensional Brownian motion, and denote by $\F := (\mathcal{F}_t)_{t \geq 0}$ a right-continuous filtration to which $B$ is adapted. We introduce the nonempty set
\begin{align} \label{setA}
\begin{split}
\mathcal{S} := & \{\D:\Omega \times \R_+ \to \R_+, \mbox{ $\F$-adapted and such that } t \mapsto \D_t \mbox{ is $\P$-a.s.}  \\
& \hspace{1cm} \mbox{non-decreasing, right-continuous and s.t. } \D_{0^-} = 0 \}, 
\end{split} 
\end{align}
For frequent future use, notice that any $\D \in \mathcal{S}$ can be expressed as (cf., e.g., \cite[VI.52–53.b]{DellacherieMeyer}),
$$\D_t = \D^c_t + \D^j_t, \qquad t \geq 0,$$
where $\D^c$ is the continuous part of $\D$, and the jump part $\D^j$ is such that 
\begin{equation} \label{DcDj}
\D^j_t:= \sum_{0 \leq s \leq t}\Delta \D_s, 
\quad \text{where} \quad \Delta\D_t := \D_{t} - \D_{t^-}, t>0, 
\quad \text{and} \quad \Delta\D_0=\D_0.
\end{equation}

We then consider on $(\Omega,\mathcal{F},\P)$ a process $X$ satisfying the following stochastic differential equation (SDE) 
\begin{equation}
\label{stateX}
\d X_t^{x,\D} =  b(X_t^{x,\D}) \d t + \ss (X_t^{x,\D}) \d B_t - \d\D_t, \quad X_{0^-}^{x,\D}  = x \in \mathcal{O}.
\end{equation}
Here, $\mathcal{O} := (\underline x, \overline x)$ with $- \infty \leq \underline x < \overline x \leq + \infty$ and assume that $\underline x$ and $\overline x$ are unattainable for the uncontrolled process $X^{x,0}$ (i.e.\ the solution to \eqref{stateX} associated to $\D \equiv 0$). 
Furthermore, $b$ and $\sigma$ satisfy the following assumption.

\begin{assumption}
\label{A1}
The coefficients $b:\mathcal{O} \to \R$ and $\sigma:\mathcal{O} \to (0,\infty)$ are such that for any $x \in \mathcal{O}$ and any $\D \in \mathcal{S}$, there exists a unique strong solution $X^{x,\D}$ to \eqref{stateX}.
\end{assumption}

In order to give more context to the aforementioned assumption we make the following remark.

\begin{remark}
\label{rem:Ass1}
If we impose the conditions that, for any $R>0$ and for some $L_R>0$, we have 
$$
|b(y) - b(x)| + |\sigma(y)-\sigma(x)|\leq L_R|x-y|, \quad \forall x,y:\,\, |x|\leq R,\, |y|\leq R,
$$
then we know from \cite[Theorem V.7]{Protter} and the discussion after its proof that, for any $\D \in \mathcal{S}$, the Assumption \ref{A1} holds true. 
These conditions are satisfied by many well-known diffusions, such as the Brownian motion, the Brownian motion with drift, the geometric Brownian motion, and the Ornstein-Uhlenbeck process. 

However, it is important to note that our subsequent results do not exploit any Lipschitz continuity properties of $b$ and $\sigma$.  
In this respect, we can also consider other important It\^o-diffusions with non-Lipschitz continuous coefficients that satisfy Assumption \ref{A1}, such as the CIR process, under the so-called Novikov condition; See also Section 4 in the recent \cite{Liu-etal} for the solution of an optimal discounted harvesting problem under general singularly controlled It\^o dynamics (which may be of CIR type; cf.~\cite[\it Example~2.12]{Liu-etal}), but without a running infimum/supremum component in the state process.
\end{remark}

Given $x, i \in \mathcal{O}$ such that $i\leq x$ and $\D \in \mathcal{S}$, we define the running infimum of $X^{x,\D}$ from~\eqref{stateX}~by
\begin{equation} \label{eq:infim}
I^{i,\D}_t:=i \wedge \inf_{0\leq s \leq t} X^{x,\D}_s, \quad t \geq 0,
\end{equation}
and we notice that the two-dimensional controlled process $(X^{x,\D}, I^{i,\D})$ is such that 
$$ 
I^{i,\D}_t \leq X^{x,\D}_t, \quad \forall\;  t\geq 0, \quad \P-\text{a.s.}
$$
It thus follows that the state-space of the process $(X^{x,\D}, I^{i,\D})$ is given by
\begin{equation}
\label{eq:setM}
\mathcal{M}:=\{(x,i) \in \mathcal{O} \times \mathcal{O}:\, x \geq i\}.
\end{equation}

\begin{remark} \label{rem:sup}
We begin with the study of optimal control problems in state-spaces defined by two-dimensional processes $(X^{x,\D}, I^{i,\D})$. 
The problem formulations in state-spaces defined by controlled diffusions and their running supremum processes are presented later on, in Section \ref{sec:sup}, since their associated results are obtained via the ones for $(X^{x,\D}, I^{i,\D})$.
\end{remark}

In the rest of the paper we shall also use the notation $\E_{x,i}[f(X^{\D}_t, I^{\D}_t)]=\E[f(X^{x,\D}_t, I^{i,\D}_t)]$. Here $\E_{x,i}$ is the expectation under the measure $\P_{x,i}(\,\cdot\,):=\P(\,\cdot\,|X^{\D}_{0^-}=x, I^{\D}_{0^-}=i)$ on $(\Omega,\mathcal{F})$, and $f:\mathcal{M} \to \R$ is any Borel-measurable function such that $\E_{x,i}[f(X^{\D}_t, I^{\D}_t)]$ is well-defined. 
Furthermore, we shall denote by $X^{x,0}$ and $I^{i,0}$ the uncontrolled solution to \eqref{stateX} (i.e.\ the solution to \eqref{stateX} obtained by taking $\D\equiv 0$) and its corresponding running infimum process.

\subsection{The controlled process $(X^\D, I^\D)$}
\label{sec:processsetting}

The challenging part in this problem formulation is that by controlling downwards the process $X^{x,\D}$ as in \eqref{stateX} via the control process $\D \in \mathcal{S}$ from \eqref{setA}, the running infimum process $I^{i,\D}$ defined by \eqref{eq:infim} is also controlled. 
Recall that the control process $\D$ is linearly controlling the process $X^{x,\D}$, $\P$--a.s., such that 
\begin{equation}
\label{XDD}
X_t^{x,\D} = x + \int_0^t b(X_s^{x,D}) \d s  + \int_0^t \ss (X_s^{x,D}) \d B_s - \D^c_t - \sum_{0 \leq s \leq t} \Delta \D_s, \quad x \in \mathcal{O},
\end{equation}
where $\D^c_t$ is the cumulative amount of control exerted from time 0 up to time $t$ via the continuous part of $\D$, and the sum accounts for the cumulative amount of control exerted from time 0 up to time $t$ via the jump part $\D^j$ of $\D$ defined by \eqref{DcDj}. 

However, contrary to $X^{x,\D}$, the process $I^{i,\D}$ is not always controlled by the control process $\D$.
Informally, the process $I^{i,\D}$ may remain unaffected by changes in $\D$ and $X^{x,\D}$, or decrease by an equal amount as the increase of $\D$ (equivalently, decrease of $X^{x,\D}$), or even decrease by a strictly smaller amount than the increase of $\D$.
These three possibilities arise essentially due to the definition of the process $I^{i,\D}$ in \eqref{eq:infim}. 
In particular, $I^{i,\D}$ is a non-increasing process, which is decreasing only along the diagonal $\partial \mathcal{M}$ of the state-space $\mathcal{M}$ from \eqref{eq:setM}, defined by
\begin{equation}
\label{eq:diag}
\partial \mathcal{M}:=\{(x,i) \in \mathcal{O} \times \mathcal{O}:\, x = i\}, 
\end{equation}
and remains constant while the process $(X^{x,\D}, I^{i,\D})$ is away from the diagonal, namely in 
$$
\mathcal{M} \setminus \partial \mathcal{M} = \{(x,i) \in \mathcal{O} \times \mathcal{O}:\, x > i\}. 
$$
In the remainder of this section, we elaborate further on the scenarios about the occurrences of the aforementioned three possibilities and explain in more detail the movements of the process $(X^{x,\D}, I^{i,\D})$ in its state-space $\mathcal{M}$ dictated by the control process $\D$ in each scenario.

\vspace{1mm}
{\it Scenario} (a). 
This is the one closest to the one-dimensional controlled diffusion setup. 
In particular, this can occur at some time $t_0$ when the process $(X^{x,\D}_{t_0-}, I^{i,\D}_{t_0-}) \in \mathcal{M} \setminus \partial \mathcal{M}$, and the increase in the control process $\D$ at time $t_0$ is either due to $\D^c_{t_0}$ or due to a relatively small-sized jump $\D^j_{t_0} < X^{x,\D}_{t_0-} - I^{i,\D}_{t_0-}$. 
In both cases, we will end up having $X^{x,\D}_{t_0} > I^{i,\D}_{t_0}$, i.e.~$(X^{x,\D}_{t_0}, I^{i,\D}_{t_0}) \in \mathcal{M} \setminus \partial \mathcal{M}$, hence the change in $(X^{x,\D}, I^{i,\D})$ due to the control process $D$ at time $t_0$ affects only the first coordinate.\footnote{\, Notice that, even if the size of the jump $\D^j_{t_0} =  X^{x,\D}_{t_0-} - I^{i,\D}_{t_0-}$, the process will end up exactly on the diagonal, i.e.~$(X^{x,\D}_{t_0}, I^{i,\D}_{t_0}) \in \partial \mathcal{M}$, but the second coordinate $I^{i,\D}$ will still not change due to the jump $\D^j_{t_0}$, so this is still covered by scenario (a). 
The second coordinate $I^{i,\D}$ may however decrease at the next instance, e.g.~due to a downward Brownian fluctuation, since $(X^{x,\D}_{t_0}, I^{i,\D}_{t_0}) \in \partial \mathcal{M}$.}
In the two-dimensional state-space, the process $(X^{x,\D}, I^{i,\D})$ therefore moves {\it downwards in the $x$-axis}, while maintaining the same coordinate value in the $i$-axis (see top right panel in Figure \ref{Fig2} for an example).

\vspace{1mm}
{\it Scenario} (b). 
This is the closest one to the setup of a two-dimensional stochastic process controlled by the same control process. 
In particular, this can occur at some time $t_0$ when the process $(X^{x,\D}_{t_0-}, I^{i,\D}_{t_0-}) \in \partial \mathcal{M}$, no matter whether the increase in the control process $\D$ at time $t_0$ is due to $\D^c_{t_0}$ or due to any size of jump $\D^j_{t_0}$. 
In either case, we will end up having  that any change in $(X^{x,\D}, I^{i,\D})$ due to the control process $D$ at time $t_0$ affects both coordinates equally and simultaneously, leading to $(X^{x,\D}_{t_0}, I^{i,\D}_{t_0}) \in \partial \mathcal{M}$. 
In the two-dimensional state-space, the process $(X^{x,\D}, I^{i,\D})$ therefore moves {\it downwards in a 45$^\circ$-angle direction along the diagonal $\partial \mathcal{M}$}, i.e.~in the south-west direction, by simultaneously decreasing the coordinate values in both the $x$-axis and $i$-axis (see bottom left panel in Figure \ref{Fig2} for an example).

\vspace{1mm}
{\it Scenario} (c). 
This is the most involved one and is not similar to any other standard framework in stochastic control theory. 
This can occur at some time $t_0$ when the process $(X^{x,\D}_{t_0-}, I^{i,\D}_{t_0-}) \in \mathcal{M} \setminus \partial \mathcal{M}$, and the increase in the control process $\D$ at time $t_0$ is due to a relatively large-sized jump $\D^j_{t_0} > X^{x,\D}_{t_0-} - I^{i,\D}_{t_0-}$. 
Such a jump $\D^j$ partly affects only the first coordinate of the process $(X^{x,\D}, I^{i,\D})$, i.e.~until $X^{x,\D}$ decreases to the value of $I^{i,\D}_{t_0-}$, and then it partly affects both coordinates of the process $(X^{x,\D}, I^{i,\D})$ for the remaining amount of the jump.  
We then end up having that such a change in $(X^{x,\D}, I^{i,\D})$ due to the control process $D$ at time $t_0$ affects fully the first coordinate and partly the second coordinate, leading to $(X^{x,\D}_{t_0}, I^{i,\D}_{t_0}) \in \partial \mathcal{M}$. 
In the two-dimensional state-space, the process $(X^{x,\D}, I^{i,\D})$ therefore moves {\it downwards in a `hockey-stick' direction, firstly in the $x$-axis and subsequently in a 45$^\circ$-angle along the diagonal $\partial \mathcal{M}$}, by initially decreasing the value in the $x$-axis, while maintaining the same coordinate value in the $i$-axis  (as in Scenario (a)), and subsequently simultaneously decreasing the coordinate values in both the $x$-axis and $i$-axis (as in Scenario (b)) 
(see bottom right panel in Figure \ref{Fig2}~for~an~example).

\begin{remark} \label{rem:int}
Even though we intuitively describe the jump $\D^j$ of the control process $\D$ in Scenario {\rm (c)} as a sequential procedure, in order to intuitively convey its effect on the process $(X^{x,\D}, I^{i,\D})$, the whole jump of size $\D^j$ occurs instantly at time $t_0$. 

To capture this feature in a rigorous mathematical way, part of our subsequent analysis is the introduction of two novel integrals in such state-spaces $\mathcal{M}$ defined by \eqref{eq:setM} 
(or $\mathcal{N}$ defined by \eqref{eq:setN} for $(X^{x,\D}, S^{s,\D})$), where the integrator is either the control process $\D$ or the controlled running infimum $I^{i,\D}$ defined by \eqref{eq:infim} 
(or the controlled running supremum $S^{s,\D}$ defined by \eqref{eq:sup}). 

These constructions are also consistent with the Hamilton-Jacobi-Bellman equation, which is fundamental in stochastic control theory, and can be used in all such problems which involve explicitly the running infimum or supremum process.
\end{remark}

\section{The Optimal Control Problem}
\label{sec:problem}

In this section we introduce a class of singular stochastic control problems in which the decision maker aims at maximising a discounted expected reward functional dynamically depending on the controlled diffusion process $X^\D$ and the running infimum process $I^\D$ as well, under a stochastic time-horizon given by the first absorption time of $X^{\D}$ at a given threshold function~of~$I^\D$. 

We aim at a general formulation of the above control problem  such that the expected reward functional consists of the sum of three integrals: 
(i) A classical integral with respect to the Lebesgue measure of an instantaneous profit function of the two-dimensional process $(X^{\D},I^{\D})$; 
(ii) An integral involving a function of the current level of $(X^{\D},I^{\D})$ integrated with respect to the (random) Borel-measure induced by the (monotone) control process $\D$;
(iii) An integral involving a function of the current level of $I^{\D}$ integrated with respect to the (random) Borel-measure induced by the (monotone) process $I^{\D}$.
We elaborate further on the generality of such a reward functional in Section~\ref{problem} below.

In order to mathematically formulate the above optimal control problem, we need to define two new integrals in the types (ii) and (iii) above, which take into account that the underlying controlled process $(X^\D,I^\D)$ is on the one hand two-dimensional, but on the other hand the state-space is restricted to a strict subset $\mathcal{M}$ of $\R^2$, and the control process $\D$ does not necessarily affect both coordinate processes at the same time (cf.~Scenarios (a)--(c) in Section \ref{sec:processsetting} and~Remark~\ref{rem:int}).
The role of these definitions is twofold: they allow for a dynamically consistent formulation of the control problem in the presence of a (possibly discontinuous) controlled running infimum, and they provide the appropriate framework for the derivation of the verification theorem in Section 4 and for the construction of a candidate value function in the application of Section 5 (see, in particular Section 5.3.1 for details).

\subsection{Integral with respect to $\D$} \label{SID}

The first integral we introduce is the one where the integrator is the control process $\D \in \mathcal{S}$. 
To account for the most general cases possible, we assume that the marginal reward per unit of control and the discount rate are both functions of the two-dimensional underlying process $(X^\D,I^\D)$. 

This is an appropriate extension (that preserves the consistency with the Hamilton-Jacobi-Bellman equation) of the integral introduced by Zhu in \cite{Zhu} for singular stochastic control problems in which the marginal reward per unit of control is a function of a linearly controlled It\^o-diffusion $X^{\D}$. In recent years, Zhu's integral has found numerous applications in the singular stochastic control literature,  
including other similar constructions of integrals in settings with multiplicative impact (see, e.g.~\cite{AlMotairiZervos,
Belarevetal1, 
Belarevetal2, 
DammannFerrari,
Dammannetal,
Ferrari, 
GZ15, 
HMP19,
LonZervos}). 
In particular, for one-dimensional controlled processes, it is defined as
\begin{align}
\label{defintegralold}
\begin{split}
\int_0^{T} e^{-\int_0^t r(X^{{\D}}_s)\d s} g\big(X^{{\D}}_t\big) \ {\circ}\ \d {\D}_t 
&:= \int_{0}^{T} e^{-\int_0^t r(X^{{\D}}_s)\d s} g\big(X^{{\D}}_t\big)\ \d {\D}^c_t \\
&\hspace{0.5cm} 
+ \sum_{t \leq T:\,\, \Delta {\D}_t \neq 0} e^{-\int_0^t r(X^{{\D}}_s)\d s} \int_{0}^{\Delta {\D}_t} g\big(X^{{\D}}_{t^-}-u\big)\ \d u, 
\end{split}
\end{align} 
from which it is evident that the second term on the right-hand side of \eqref{defintegralold} takes care of the integration over the jump times of the control process $\D$, which in turn yield jumps of the controlled process $X^{\D}$.
However, when the state space is enlarged to include the (controlled) running extremum process, the above formulation is no longer directly applicable. 
For instance, jumps of the control may induce discontinuities not only in the controlled diffusion, but also in the running extremum process. 
As a result, the cost contributions associated with the jump parts of the control must be evaluated in a way that explicitly accounts for their interaction with the controlled running extremum.

Below is our definition in the two-dimensional setup involving a diffusion and its running infimum processes. 
This essentially addresses the fact that all three Scenarios (a)--(c) of Section \ref{sec:processsetting} can be induced by a control process $\D$ depending on the location of the process $(X^\D,I^\D)$ in the state-space (see Figure \ref{Fig2} for such examples). 
Furthermore, it takes care of the resulting peculiar `hockey-stick' movement of the controlled process $(X^\D,I^\D)$ described in the Scenario (c) (see also Remark \ref{rem:int} and the bottom right panel in Figure \ref{Fig2} for an example).

\begin{definition} \label{int:D}
For the stochastic process $(X^{x,\D}, I^{i,\D})$ defined by \eqref{stateX}--\eqref{eq:infim} and its state-space $\mathcal{M}$ given by \eqref{eq:setM}, for any $T>0$ and generic functions $r, g: \mathcal{M} \to \R$, we define  
\begin{align}
\label{defintegral}
\begin{split}
&\int_0^{T} e^{-\int_0^t r(X^{\D}_s, I^{\D}_s)\d s} g\big(X^{\D}_t, I^{\D}_t\big) \ {\diamond}\ \d \D_t \\
& :=\int_{0}^{T} e^{-\int_0^t r(X^{\D}_s, I^{\D}_s)\d s} g\big(X^{\D}_t, I^{\D}_t\big)\ \d \D^c_t \\
& \hspace{0.5cm} \vspace{-5mm} 
+ \sum_{t \leq T:\,\, \Delta \D_t \neq 0} e^{-\int_0^t r(X^{\D}_s, I^{\D}_s)\d s} \int_0^{(X^{\D}_{t^-} - I^{\D}_{t^-}) \wedge \Delta \D_t} g\big(X^{\D}_{t^-}-u, I^{\D}_{t^-}\big)\ \d u \\
& \hspace{0.5cm} \vspace{-5mm} 
+ \sum_{t \leq T:\,\, \Delta \D_t \neq 0} e^{-\int_0^t r(X^{\D}_s, I^{\D}_s)\d s} \int_{X^{\D}_{t^-} - I^{\D}_{t^-}}^{\Delta \D_t}\mathds{1}_{\{\Delta \D_t > X^{\D}_{t^-} - I^{\D}_{t^-}\}} \, g\big(X^{\D}_{t^-}-u, X^{\D}_{t^-}-u\big)\ \d u. 
\end{split}
\end{align}     
\end{definition}
On the right-hand side of \eqref{defintegral}, the first integral with respect to the continuous (random) Borel-measure $\d \D^c_\cdot(\omega)$ is intended in the Lebesgue-Stieltjes sense, whereas the sum is defined in the absolute convergence sense.
It is noteworthy to mention here that if the running reward function $g(x,i)=g(x)$ in Definition \ref{int:D}, i.e.~depends only on the coordinate process $X^\D$, then the integral with respect to the control process $\D$ reduces to 
\begin{align}
\label{defintegralX}
\begin{split}
\int_0^{T} e^{-\int_0^t r(X^{\D}_s, I^{\D}_s)\d s} g\big(X^{\D}_t\big) \ {\diamond}\ \d \D_t 
&= \int_{0}^{T} e^{-\int_0^t r(X^{\D}_s, I^{\D}_s)\d s} g\big(X^{\D}_t\big)\ \d \D^c_t \\
&\hspace{0.5cm} \vspace{-5mm} 
+ \sum_{t \leq T:\,\, \Delta \D_t \neq 0} e^{-\int_0^t r(X^{\D}_s, I^{\D}_s)\d s} \int_{0}^{\Delta \D_t} g\big(X^{\D}_{t^-}-u\big)\ \d u, 
\end{split}
\end{align}
which is consistent with the integral definition in \eqref{defintegralold}, even if the discount rate function $r$ still depends on the two-dimensional process $(X^\D,I^\D)$. 
Definition \ref{int:D} coincides with the definition in \eqref{defintegralold} in the particular case when both the running reward function and the discount rate depend only on $X^\D$, i.e.~$g(x,i)=g(x)$ and $r(x,i)=r(x)$.

\subsection{Integral with respect to $I^\D$} \label{SII}

The second integral we define in this paper is novel in the literature. 
In particular, this is the integral where the integrator is the controlled infimum process $I^\D$, when the control $\D \in \mathcal{S}$. 
We again allow for the most general setup with the discount rate being a function of the two-dimensional underlying process $(X^\D,I^\D)$ and the marginal reward per unit of decrease of $I^\D$ being a function of the process $I^\D$ itself. 
The following definition takes into account also the peculiarity that the process $\D$ controls $I^\D$ only when $(X^\D,I^\D)$ is at the diagonal of the state space $\partial\mathcal{M}$ (cf.~Scenarios (a)--(c) in Section \ref{sec:processsetting}). 

\begin{definition} \label{int:I}
For the stochastic process $(X^{x,\D}, I^{i,\D})$ defined by \eqref{stateX}--\eqref{eq:infim} and its state-space $\mathcal{M}$ given by \eqref{eq:setM}, for any $T>0$ and generic functions $r, g: \mathcal{M} \to \R$, we define 
\begin{align}
\label{defintegral-I}
\begin{split}
& \int_0^{T} e^{-\int_0^t r(X^{\D}_s, I^{\D}_s)\d s} g\big(X^{\D}_t, I^{\D}_t\big) \ {\scriptstyle{\Box}}\ \d I^{\D}_t \\
&:= \int_{0}^{T} e^{-\int_0^t r(X^{\D}_s, I^{\D}_s)\d s} g\big(X^{\D}_t, I^{\D}_t\big)\ \d I^{\D,c}_t \\
& \hspace{0.5cm} - \sum_{t \leq T:\,\, \Delta \D_t \neq 0} e^{-\int_0^t r(X^{\D}_s, I^{\D}_s)\d s}  
\int_{X^{\D}_{t^-} - I^{\D}_{t^-}}^{\Delta \D_t} {\mathds{1}_{\{\Delta \D_t > X^{\D}_{t^-} - I^{\D}_{t^-}\}}}\, g\big(X^{\D}_{t^-}-u, X^{\D}_{t^-}-u\big)\ \d u.
\end{split}
\end{align}
\end{definition}

The first integral on the right-hand side of \eqref{defintegral-I} is intended in the Lebesgue-Stieltjes sense with respect to the (random) measure $\d I^{\D,c}_\cdot(\omega)$ induced by the continuous part $I^{\D,c}$ of the non-increasing process $I^{\D}$, and the sum is defined in the absolute convergence sense.
Notice that, given that $I^{\D}$ is a non-increasing process, the increments of its continuous part $I^{\D,c}$ are negative, and the measure $\d I^{\D,c}$ has support on $\{t\geq0 : X^{\D}_t(\omega) = I^{\D}_t(\omega)\}$, $\omega \in \Omega$ .

\begin{remark}
\label{rem:integralI}
Integrals with respect to a running infimum process $I$ (or equivalently running supremum $S$) have appeared, for example, in optimal stopping problems (see \cite{GapRod14, GapRod16a, GapRod16b, GuoZerv10, KyprOtt14, Ott13, Peskir98, RodZerv17}, among others). In those papers, process $I$ (or $S$) is continuous and uncontrolled, so that the integral is defined in the classical Lebesgue-Stieltjes sense. This is clearly embedded in our \eqref{defintegral-I} if one restricts the set of policies to include only the control processes $\D$ for which $I^{\D}$ is continuous (e.g., absolutely continuous controls).
\end{remark}


\subsection{Stochastic control problem formulation}
\label{problem}

Given measurable functions $r:\mathcal{M} \to \R$, $\Pi:\mathcal{M} \to \R$, $f:\mathcal{M} \to \R$, and $h:\R \to \R$, 
we consider a decision maker who discounts future rewards, at any $t\geq 0$, with a stochastic rate $r(X^\D_t, I^\D_t)$, and receives an instantaneous stochastic reward $\Pi(X^\D_t, I^\D_t)$ per unit of time. 
For each time $t\geq 0$ that the decision maker exerts control $\D$ to decrease the process $X^\D$ (and occasionally $I^\D$; cf.~Section \ref{sec:processsetting}), the decision maker receives a marginal reward $f(X^\D_t, I^\D_t)$ per unit of exerted control. 
Finally, we further allow for the possibility of a marginal reward $h(I^\D_t)$ per unit of decrease of the running infimum process $I^\D_t$. 
Given that $I^\D$ decreases only when the process $(X^\D, I^\D)$ moves on the diagonal $\partial\mathcal{M}$ of the state-space, i.e.~when $X^\D=I^\D$, the latter reward is taken without loss of generality to be a function of $I^\D$ only.  
In such a setup, the aim of the decision maker is to select an optimal control strategy $\D^{\star}$ that maximises their collective rewards.

Clearly, if any of the aforementioned reward functions $\Pi, f, h$, is negative on some subset of $\mathcal{M}$, then the corresponding negative reward function represents a cost incurred by the decision maker. Furthermore, each such reward function $\Pi, f, h$ can be taken identically equal to zero, depending on the application under consideration.
There is a plethora of applications in the literature of singular control problems with different types of trade-offs for decision makers, that can be embedded in our present setup. 
In Section \ref{sec:optdiv}, we present one such application of the traditional optimal dividend problem with the novelty that the worst performance of the company’s surplus process impacts the manager's time-preference for paying dividends. 

For any starting value $X^\D_{0^-}=x$ of the underlying process \eqref{stateX} and $I^\D_{0^-}=i$ of its running infimum \eqref{eq:infim}, such that $(x,i)\in \mathcal{M}$, we are now ready to define in mathematical terms the expected reward functional $\mathcal{J}_{x,i}$ of the decision maker (based on the aforementioned rewards)~by 
\begin{align}
\label{eq:payoff}
\begin{split}
\mathcal{J}_{x,i}(\D) := &\E_{x,i}\bigg[ \int_0^{\tau_a^{\D}} e^{-\int_0^t r(X^{\D}_s, I^{\D}_s) \, \d s} \, \Pi\big(X^{\D}_t, I^{\D}_t\big)\ \d t \bigg]\\
&+ \E_{x,i}\bigg[ \int_0^{\tau_a^{\D}} e^{-\int_0^t r(X^{\D}_s, I^{\D}_s)\, \d s} \Big(  f\big(X^{\D}_t, I^{\D}_t\big) \ {\diamond}\ \d \D_t + h\big(I^{\D}_t\big) \ {\scriptstyle{\Box}}\ \d I^{\D}_t \Big)\bigg],
\end{split}
\end{align}
where, for a given function $a:\R \to \R$ such that $a \in C(\R;\R)$, and any $\D \in \mathcal{S}$, we set
\begin{equation}
\label{tau}
\tau_a^{\D} 
= \tau_a^{\D}(x,i) 
:= \inf\big\{t \geq 0:\, X^{x,\D}_t \leq a(I^{i,\D}_t) \big\}, \quad (x,i)\in \mathcal{M}. 
\end{equation}
We shall refer to $\tau_a^{\D}$ from \eqref{tau} as the \emph{absorption time} of the process $(X^{x,\D}, I^{i,\D})$ at which the problem terminates, which can of course be infinite (see, e.g.~Lemma \ref{lem:state} below). 

Then, for $(x,i)\in \mathcal{M}$, we introduce the \emph{set of admissible controls} for the decision maker as
\begin{eqnarray}
\label{set:A}
& \hspace{0.5cm}\mathcal{A}(x,i) := 
\bigg\{ \D \in \mathcal{S}:\, X^{x,\D}_{t} \in \mathcal{O},\,\, X^{x,\D}_{t^-} - \Delta \D_t \geq a\big(I^{i,\D}_{t^-} - \big(\Delta \D_t - (X^{x,\D}_{t^-} - I^{i,\D}_{t^-})\big)^+\big) \vee \underline{x} \\
&\text{and}\,\, \E_{x,i}\bigg[\int_0^{\tau_a^{\D}} e^{-\int_0^t r(X^{\D}_s, I^{\D}_s)\, \d s} \Big(\big|\Pi\big(X^{\D}_t, I^{\D}_t\big)\big| \, \d t +  \big|f\big(X^{\D}_t, I^{\D}_t\big)\big| \, {\diamond}\, \d \D_t + \big|h\big(I^{\D}_t\big)\big| \, {\scriptstyle{\Box}}\, \d I^{\D}_t\Big)\bigg] < \infty \bigg\}. \nonumber
\end{eqnarray}
An admissible control $\D$ is therefore a process belonging to the set $\mathcal{S}$ and cannot take the process $X^\D$ strictly below $a(I^\D)$ or $\underline{x}$, i.e.~exiting the state-space cannot be caused by a jump of the control variable.
To see this more clearly, recall that the quantity $\big(\Delta \D_t - (X^{x,\D}_{t^-} - I^{i,\D}_{t^-})\big)^+$ represents the size of the jump of the infimum process caused by a jump of the control of amplitude $\Delta \D_t$ (cf.~Scenarios (a)--(c) in Section \ref{sec:processsetting}). 
Furthermore, for any $\D \in \mathcal{A}(x,i)$ one has that $\mathcal{J}_{x,i}(\D)$ is finite.

In view of the absorption time $\tau_a^{\D}$ from \eqref{tau} modelling the stochastic time-horizon of the expected reward functional $\mathcal{J}_{x,i}$ in \eqref{eq:payoff}, we consider the restricted version $\mathcal{M}^a$ of the state-space \eqref{eq:setM}, given by
\begin{equation}
\label{eq:setMa}
\mathcal{M}^a 
:= \{(x,i) \in \mathcal{M} :\, x \geq a(i) \} .
\end{equation}

\begin{lemma} \label{lem:state}
For any $(x,i) \in \mathcal{M}$ and $\D \in \mathcal{A}(x,i)$, we note that: 

\begin{enumerate}[\rm (i)]
\item If $x>a(i)$, then the state-space of the process restricts to the set $\mathcal{M}^a$ defined by \eqref{eq:setMa}. 

\item If $x \leq a(i)$, then 
\begin{equation*} 
\tau_a^{\D}(x,i) = 0 \quad \text{and} \quad 
\mathcal{A}(x,i) = \emptyset .
\end{equation*}

\item If $(a(y),y) \not\in \mathcal{M}$ for all $y \in (\underline{x}, i]$, then 
$$
\tau_a^{\D}(x,i) = \infty, \; \P-\text{a.s.}, \quad \text{and} \quad \mathcal{M}^a \cap \{(x,y) \in \mathcal{M} \;|\; y \leq i \} \equiv \mathcal{M} \cap \{(x,y) \in \mathcal{M} \;|\; y \leq i \}.
$$
\end{enumerate}
\end{lemma}
\begin{proof}
It follows from the definition \eqref{eq:payoff} of the optimisation criterion $J_{x,i}$ and the definition \eqref{tau} of $\tau^\D_a$ that, if $(x,i) \in \mathcal{M}$ is such that $x > a(i)$, then for all admissible controls $\D \in \mathcal{A}(x,i)$ the absorption time from \eqref{tau} takes the form 
\begin{equation*}
\tau_a^{\D}(x,i) = \inf\big\{t \geq 0:\, X^{x,\D}_t \leq a(I^{i,\D}_t) \big\}.
\end{equation*}
This proves part (i). 

Part (ii) follows straightforwardly also from the definition \eqref{set:A} of the admissible set $\mathcal{A}(x,i)$ and the definition \eqref{tau} of $\tau^\D_a$. 
Combining these with the observation that if $(a(y),y) \not\in \mathcal{M}$ for all $y \in (\underline{x}, i]$, then for all $(x,y) \in \mathcal{M}$ such that $y \leq i$, we have $x \geq y > a(y)$. 
Thus, we conclude that part (iii) holds true as well.  
\end{proof}

In order to ensure that the trivial control strategy $\D\equiv 0$ of `doing nothing', i.e.~simply observing and waiting forever,  belongs to $\mathcal{A}(x,i)$ (and therefore that $\mathcal{A}(x,i) \neq \emptyset$) we make the following assumption.
\begin{assumption} \label{ass:noemptyA} 
The state-space process $(X^0,I^0)$ and the functions $a$, $r$, $\Pi$ and $h$ satisfy
$$
\E_{x,i}\bigg[\int_0^{\tau_a^{0}} e^{-\int_0^t r(X^{0}_s, I^{0}_s)\, \d s} \Big(\big|\Pi\big(X^{0}_t, I^{0}_t\big)\big| \, \d t + \big|h\big(I^{0}_t\big)\big| \, {\scriptstyle{\Box}}\, \d I^{0}_t\Big)\bigg] < \infty
$$
\end{assumption}

For any $(x,i) \in \mathcal{M}$, the decision maker thus aims at determining an optimal, admissible control process $\D^{\star} \in \mathcal{A}(x,i)$ that maximises $\mathcal{J}_{x,i}$ as in \eqref{eq:payoff}; hence, at solving
\begin{equation}
   \label{eq:value}
   V(x,i):=\sup_{\D \in \mathcal{A}(x,i)}\mathcal{J}_{x,i}(\D).
\end{equation}
In the sequel, we shall refer to $V$ as the `value function'.


\section{The Hamilton-Jacobi-Bellman Equation and the Verification Theorem}
\label{sec:verthm}

We begin this section, by introducing the infinitesimal generator $\mathcal{L}_{X}$ of the uncontrolled process $X^{x,0}$, i.e.~the solution to the SDE \eqref{stateX} when $\D\equiv 0$, acting on functions belonging to $C^2({\mathcal{O}})$, and defined by
\begin{align}
\label{eq:LX}
\mathcal{L}_{X}:=\frac{1}{2}\sigma^2(x)\frac{\d^2}{\d x^2} +\mu(x)\frac{\d}{\d x}.
\end{align}
In the following, to simplify notation, we write $g_x$, $g_{xx}$ for a generic function $g: \mathcal{M} \to \R$, to denote its first and second derivatives with respect to $x$, and $g_i$ to denote its first derivative with respect to the variable $i$.

For our forthcoming analysis we further define the $i$-section of $\mathcal{M}^a$ from \eqref{eq:setMa} by
\begin{equation}
\label{eq:setMi}
\mathcal{M}^a_i:=\{x \in \mathcal{O}:\, (x,i) \in \mathcal{M}^a\},
\end{equation}
which will play an important role.

By the dynamic programming principle, we expect that the value function $V$ from \eqref{eq:value} identifies with the solution (in a suitable sense) to the Hamilton-Jacobi-Bellman (HJB) equation 
\begin{equation} \label{eq:HJB}
\max\Big\{\big[(\mathcal{L}_{X} -r(x,i))v\big](x,i) + \Pi(x,i), \, f(x,i) - v_x(x,i)\Big\} = 0, 
\quad x \in \mathcal{M}^a_i \setminus \{i, a(i)\}, 
\,\, i \in \mathcal{O},
\end{equation}
subject to the boundary conditions
\begin{equation}
\label{eq:bdconds1}
v_i(x,i)|_{x=i} = h(i), \quad i \in \mathcal{O} \quad \text{s.t.} \quad a(i) < i  \quad \text{(Neumann boundary condition)},  
\end{equation}
and  
\begin{equation}
\label{eq:bdconds2}
v(a(i),i) =0, \quad i \in \mathcal{O} \quad \text{s.t.} \quad a(i) \geq i \quad \text{(Dirichlet boundary condition)}.
\end{equation}

We now present the main result of this section, which is a verification theorem providing sufficient conditions for constructing an optimal control strategy for the general singular control problem in \eqref{eq:value}.

\begin{theorem}[Verification Theorem]
\label{thm:verification}
Recall the value function $V$ from \eqref{eq:value} and fix an arbitrary pair $(x,i) \in \mathcal{M}$. 
Then, we have the following results:
\begin{itemize}
\item[(i)] Let $v$ be a $C^{2,1}(\mathcal{M}^a)$-solution to \eqref{eq:HJB} satisfying \eqref{eq:bdconds1} and \eqref{eq:bdconds2}. 
If also 
\begin{equation} \label{eq:integrability} 
\E_{x,i}\bigg[\sup_{t \in [0,\tau^{\D}_a]} e^{-\int_0^{t} r(X^{\D}_s,I^{\D}_s) \, \d s} \, \big|v\big(X^{\D}_{t},I^{\D}_{t}\big)\big|\bigg] < + \infty , 
\end{equation}
and, for $\D \in \mathcal{A}(x,i)$ such that $\P_{x,i}(\tau_a^{\D}=+\infty)>0$, we further have
\begin{equation} \label{eq:transversality}
\limsup_{T \uparrow \infty}\E_{x,i}\Big[e^{-\int_0^{T} r(X^{\D}_s, I^{\D}_s) \, \d s } \, v(X^{\D}_{T}, I^{\D}_T) \, \mathds{1}_{\{T < \tau^{\D}_a\}}\Big] = 0,
\end{equation}
then $v \geq V$ on $\mathcal{M}^a$.

\item[(ii)] Suppose also that there exists $\D^{\star} \in \mathcal{A}(x,i)$ such that 
on $\{s < \tau^{\D^{\star}}_a \}$ we have 
\begin{equation} \label{eq:Skh1}
\big[\big(\mathcal{L}_X - r(X^{\D^{\star}}_s, I^{\D^{\star}}_s)\big)v\big](X^{\D^{\star}}_s, I^{\D^{\star}}_s) + \Pi (X^{\D^{\star}}_s, I^{\D^{\star}}_s) = 0, \quad \P_{x,i} \otimes \d s - \text{almost everywhere},
\end{equation}
and 
\begin{equation} \label{eq:Skh2}
\int_{[0,\tau^{\D^{\star}}_a]} \mathds{1}_{\{v_x(X^{\D^{\star}}_u, I^{\D^{\star}}_u) > f(X^{\D^{\star}}_u, I^{\D^{\star}}_u)\}} \ \d \D^{\star}_u = 0 , \quad 
\P_{x,i} - \text{almost surely}.
\end{equation}
Then, $v=V$ on $\mathcal{M}^a$ and $\D^{\star}$ is optimal for \eqref{eq:value}.
    \end{itemize}
\end{theorem}

\begin{proof}
We fix an arbitrary pair $(x,i) \in \mathcal{M}^a$ and prove the two parts separately. 

\vspace{1mm}
{\it Proof of part }(i). 
Suppose that $\D \in \mathcal{A}(x,i)$. 
Then define the sequence $(\zeta_n)_{n\in \mathbb{N}}$ of localising stopping times
$$
\zeta_n
:= \inf\Big\{t \geq 0:\, \int_0^t e^{- 2\int_0^s r(X^{\D}_u,I^{\D}_u) \, \d u} \, \sigma^2(X^{\D}_s) \,  |v_x(X^{\D}_s,I^{\D}_s)|^2 \ \d s \geq n\Big\}, \quad n \in \mathbb{N}, \quad \P_{x,i}-\text{a.s.},$$
that makes the local martingale
$$M:= \Big\{\int_0^t e^{-\int_0^s r(X^{\D}_u,I^{\D}_u)\, \d u} \, \sigma(X^{\D}_s) \, v_x(X^{\D}_s,I^{\D}_s) \ \d W_s\Big\}_{t\geq0}$$
a true $\mathbb{F}$-martingale under $\P_{x,i}$, which then has zero expectation.

Then, for $T>0$, applying It\^o-Meyer's formula for semimartingales, and taking expectations under $\P_{x,i}$ we obtain
\begin{align}
    \label{eq:Ito1}
    & \E_{x,i}\Big[e^{-\int_0^{T \wedge \zeta_n \wedge \tau_a^{\D}} r(X^{\D}_s,I^{\D}_s) \, \d s} \, v\big(X^{\D}_{T \wedge \zeta_n \wedge \tau_a^{\D}},I^{\D}_{T \wedge \zeta_n \wedge \tau_a^{\D}}\big)\Big] 
    = v(x,i) \nonumber \\
    & + \E_{x,i}\bigg[\int_0^{T \wedge \zeta_n \wedge \tau_a^{\D}} e^{-\int_0^{t} r(X^{\D}_s,I^{\D}_s) \, \d s} \, \big[\big(\mathcal{L}_X - r(X^{\D}_t,I^{\D}_t)\big)v\big](X^{\D}_t,I^{\D}_t) \ \d t\bigg]  \nonumber \\
    & + \E_{x,i}\bigg[\int_0^{T \wedge \zeta_n \wedge \tau_a^{\D}} e^{-\int_0^{t} r(X^{\D}_s,I^{\D}_s) \, \d s} \Big(v_i(X^{\D}_t,I^{\D}_t) \ \d I^{\D,c}_t - v_x(X^{\D}_t,I^{\D}_t) \ \d \D^c_t\Big) \bigg] \nonumber \\
    & + \E_{x,i}\bigg[\sum_{t \leq T \wedge \zeta_n \wedge \tau_a^{\D}:\, \Delta \D_t \neq 0} e^{-\int_0^{t} r(X^{\D}_s,I^{\D}_s) \, \d s} \Big(v(X^{\D}_t,I^{\D}_t) - v(X^{\D}_{t^-},I^{\D}_{t^-})\Big)\bigg].
\end{align}

Let us now focus our attention on the sum term appearing in the last formula. It is here that Definitions \ref{int:D} and \ref{int:I} play an important role. 
In particular, we have $\P_{x,i}$-a.s. that 
\begin{align}
& v(X^{\D}_t,I^{\D}_t) - v(X^{\D}_{t^-},I^{\D}_{t^-}) \nonumber \\
&= v\big(X^{\D}_{t^-} - \Delta \D_t,I^{\D}_{t^-} - \big(\Delta \D_t - (X^{\D}_{t^-} - I^{\D}_{t^-})\big)^+\big) - v(X^{\D}_{t^-},I^{\D}_{t^-}) \nonumber \\
& = \int_0^{\Delta \D_t} \frac{\d v}{\d u}\big(X^{\D}_{t^-} - u,I^{\D}_{t^-} - \big(u - (X^{\D}_{t^-} - I^{\D}_{t^-})\big)^+\big) \ \d u \nonumber \\
& = - \int_0^{\Delta \D_t} \Big[   v_x\big(X^{\D}_{t^-} - u,I^{\D}_{t^-} - \big(u - (X^{\D}_{t^-} - I^{\D}_{t^-})\big)^+\big) \,+\, 
\mathds{1}_{\{u > X^{\D}_{t^-} - I^{\D}_{t^-} \}} \, v_i\big(X^{\D}_{t^-} -u, X^{\D}_{t^-}-u\big)\Big] \ \d u \nonumber \\
&= - \int_0^{(X^{\D}_{t^-} - I^{\D}_{t^-}) \wedge \Delta \D_t} v_x\big(X^{\D}_{t^-} - u,I^{\D}_{t^-}\big) \ \d u \nonumber \\
&\quad - \int_{X^{\D}_{t^-} - I^{\D}_{t^-}}^{\Delta \D_t} \big(v_x + v_i\big)\big(X^{\D}_{t^-} - u, X^{\D}_{t^-} - u\big) \, \mathds{1}_{\{\Delta \D_t > X^{\D}_{t^-} - I^{\D}_{t^-} \}}\ \d u. \nonumber
\end{align}
Upon inserting the last equation into \eqref{eq:Ito1} and using definitions \eqref{defintegral} and \eqref{defintegral-I} we obtain
\begin{align}
    \label{eq:Ito1bis}
    & \E_{x,i}\Big[e^{-\int_0^{T \wedge \zeta_n \wedge \tau_a^{\D}} r(X^{\D}_s,I^{\D}_s) \, \d s} \, v\big(X^{\D}_{T \wedge \zeta_n \wedge \tau_a^{\D}},I^{\D}_{T \wedge \zeta_n \wedge \tau_a^{\D}}\big)\Big] 
    = v(x,i)  \nonumber \\
    & + \E_{x,i}\bigg[\int_0^{T \wedge \zeta_n \wedge \tau_a^{\D}} e^{-\int_0^{t} r(X^{\D}_s,I^{\D}_s) \, \d s} \big[\big(\mathcal{L}_X - r(X^{\D}_t,I^{\D}_t)\big)v\big](X^{\D}_t,I^{\D}_t) \ \d t\bigg]  \nonumber \\
    & + \E_{x,i}\bigg[\int_0^{T \wedge \zeta_n \wedge \tau_a^{\D}} e^{-\int_0^{t} r(X^{\D}_s,I^{\D}_s)\, \d s} \Big(v_i(X^{\D}_t,I^{\D}_t) \ {\scriptstyle{\Box}}\ \d I^{\D}_t - v_x(X^{\D}_t,I^{\D}_t) \ {\diamond}\ \d \D_t\Big) \bigg].
\end{align}

Hence, employing \eqref{eq:HJB}, the fact that the (random) measure $\d I^{\D}_{\cdot}(\omega)$ has support on $\{t\geq0: X^{\D}_t(\omega) = I^{\D}_t(\omega)\}$, $\omega \in \Omega$, and then the boundary condition \eqref{eq:bdconds1}, yields
\begin{align*}
    & \E_{x,i}\Big[e^{-\int_0^{T \wedge \zeta_n \wedge \tau_a^{\D}} r(X^{\D}_s,I^{\D}_s) \, \d s} \, v\big(X^{\D}_{T \wedge \zeta_n \wedge \tau_a^{\D}},I^{\D}_{T \wedge \zeta_n \wedge \tau_a^{\D}}\big)\Big] 
    \leq v(x,i) \nonumber \\
    & - \E_{x,i}\bigg[\int_0^{T \wedge \zeta_n \wedge \tau_a^{\D}} e^{-\int_0^{t} r(X^{\D}_s,I^{\D}_s) \, \d s} \Big(\Pi(X^{\D}_t,I^{\D}_t)\ \d t + f(X^{\D}_t,I^{\D}_t) \ {\diamond}\ \d \D_t + h(I^{\D}_t) \ {\scriptstyle{\Box}}\ \d I^{\D}_t\Big) \bigg];
\end{align*}
that is, rearranging terms,
\begin{align}
\label{eq:Ito1tris}
v(x,i) &\geq \E_{x,i}\Big[e^{-\int_0^{T \wedge \zeta_n \wedge \tau_a^{\D}} r(X^{\D}_s,I^{\D}_s)\, \d s} \, v\big(X^{\D}_{T \wedge \zeta_n \wedge \tau_a^{\D}},I^{\D}_{T \wedge \zeta_n \wedge \tau_a^{\D}}\big)\Big] \nonumber \\
&\quad + \E_{x,i}\bigg[\int_0^{T \wedge \zeta_n \wedge \tau_a^{\D}} \hspace{-3mm} e^{-\int_0^{t} r(X^{\D}_s,I^{\D}_s) \, \d s} \Big(\Pi(X^{\D}_t,I^{\D}_t) \, \d t + f(X^{\D}_t,I^{\D}_t) \, {\diamond}\, \d \D_t + + h(I^{\D}_t)  \, {\scriptstyle{\Box}}\, \d I^{\D}_t\Big) \bigg].
\end{align}

Notice now that by \eqref{eq:integrability} we obtain  
\begin{align}
\label{eq:integrabilityIto}
\begin{split}
&e^{-\int_0^{T \wedge \zeta_n \wedge \tau_a^{\D}} r(X^{\D}_s,I^{\D}_s)\, \d s} \, \big|v\big(X^{\D}_{T \wedge \zeta_n \wedge \tau_a^{\D}},I^{\D}_{T \wedge \zeta_n \wedge \tau_a^{\D}}\big)\big| \\
&\qquad \qquad \qquad \qquad \leq \sup_{t \in [0,\tau^{\D}_a]} e^{-\int_0^{t} r(X^{\D}_s,I^{\D}_s) \, \d s} \, \big|v\big(X^{\D}_{t},I^{\D}_{t}\big)\big| \in L^1(\Omega, \mathcal{F}, \P_{x,i})
\end{split}
\end{align}
and also by \eqref{set:A} we know that
$$ 
\E_{x,i}\bigg[\int_0^{\tau_a^{\D}} e^{-\int_0^{t} r(X^{\D}_s,I^{\D}_s) \, \d s} \Big( \big|\Pi(X^{\D}_t,I^{\D}_t)\big| \, \d t + \big|f(X^{\D}_t,I^{\D}_t)\big|  \, {\diamond}\, \d \D_t + \big|h(I^{\D}_t)\big|  \, {\scriptstyle{\Box}}\, \d I^{\D}_t\Big)\bigg] < \infty.$$
Hence, we can apply the dominated convergence theorem to let first $n\uparrow \infty$ to obtain
\begin{align*}
v(x,i) &\geq \E_{x,i}\Big[e^{-\int_0^{T \wedge \tau_a^{\D}} r(X^{\D}_s,I^{\D}_s) \, \d s} \, v\big(X^{\D}_{T \wedge \tau_a^{\D}}, I^{\D}_{T \wedge \tau_a^{\D}}\big)\Big] \nonumber \\
&\quad + \E_{x,i}\bigg[\int_0^{T \wedge \tau_a^{\D}} e^{-\int_0^{t} r(X^{\D}_s,I^{\D}_s) \, \d s} \Big(\Pi(X^{\D}_t,I^{\D}_t) \, \d t + f(X^{\D}_t,I^{\D}_t)  \, {\diamond}\, \d \D_t +  h(I^{\D}_t)  \, {\scriptstyle{\Box}}\, \d I^{\D}_t\Big) \bigg].
\end{align*}

Then, we observe that 
\begin{align*}
&\E_{x,i}\Big[e^{-\int_0^{T \wedge \tau_a^{\D}} r(X^{\D}_s,I^{\D}_s)\, \d s} \, v\big(X^{\D}_{T \wedge \tau_a^{\D}}, I^{\D}_{T \wedge \tau_a^{\D}}\big)\Big] \nonumber \\
&=\E_{x,i}\Big[e^{-\int_0^{T} r(X^{\D}_s,I^{\D}_s) \, \d s} \,  v\big(X^{\D}_{T}, I^{\D}_{T}\big) \, \mathds{1}_{\{\tau^{\D}_a > T\}} \Big]
+ \E_{x,i}\Big[e^{-\int_0^{\tau_a^{\D}} r(X^{\D}_s,I^{\D}_s) \, \d s} \,  v\big(X^{\D}_{\tau_a^{\D}}, I^{\D}_{\tau_a^{\D}}\big) \, \mathds{1}_{\{\tau_a^{\D} \leq T\}} \Big] \nonumber \\
&= \E_{x,i}\Big[e^{-\int_0^{T} r(X^{\D}_s,I^{\D}_s) \, \d s} \, v\big(X^{\D}_{T}, I^{\D}_{T}\big) \,  \mathds{1}_{\{\tau^{\D}_a > T\}} \Big],
\end{align*}
where in the last equality we used the definition \eqref{tau} of $\tau^{\D}_a$ and \eqref{eq:bdconds2} which implies that $v\big(X^{\D}_{\tau^{\D}_a}, I^{\D}_{\tau^{\D}_a}\big) = 0$, $\P_{x,i}$-a.s.\ on $\{\tau^{\D}_a \leq T\}$. 
Then, upon using also \eqref{eq:transversality}, we can apply once more the dominated convergence theorem (again thanks to \eqref{eq:integrabilityIto}) to let $T\uparrow \infty$ and get
\begin{align}
\label{eq:Itofinal}
& v(x,i) \geq \E_{x,i}\bigg[\int_0^{\tau_a^{\D}} e^{-\int_0^{t} r(X^{\D}_s,I^{\D}_s) \, \d s} \Big(\Pi(X^{\D}_t,I^{\D}_t) \, \d t + f(X^{\D}_t,I^{\D}_t)  \, {\diamond}\, \d \D_t + h(I^{\D}_t)  \, {\scriptstyle{\Box}}\, \d I^{\D}_t\Big) \bigg].
\end{align}
Notice indeed that \eqref{eq:transversality} is automatically satisfied thanks to \eqref{eq:integrabilityIto} when $\P_{x,i}(\tau^{\D}_a < +\infty)=1$ and it is actually needed to hold only for those $\D \in \mathcal{A}(x,i)$ for which $\P_{x,i}(\tau^{\D}_a = +\infty)>0$, as required.

Hence, \eqref{eq:Itofinal} implies that $v(x,i) \geq \mathcal{J}_{x,i}(\D)$ for any $\D \in \mathcal{A}(x,i)$ and any $(x,i)\in \mathcal{M}^a$. By arbitrariness, $v\geq V$ on $\mathcal{M}^a$, which completes the proof of part (i).

\vspace{1mm}
{\it Proof of part }(ii). 
We now want to prove that actually $v=V$ on $\mathcal{M}^a$ and that $D^{\star}$ defined as in the statement of part (ii) of the theorem is indeed optimal. 
To that end, notice that for $D^{\star}$ all the inequalities leading to \eqref{eq:Itofinal} are in fact equalities. Therefore, 
$$v(x,i) = \mathcal{J}_{x,i}(\D^{\star}) \leq V(x,i).$$
Given the arbitrariness of $(x,i) \in \mathcal{M}^a$, we conclude that $v\leq V$ on $\mathcal{M}^a$, which, together with the previous step, implies that $v=V$ and that $D^{\star}$ is optimal.
\end{proof}

\begin{remark} \label{comments}
Theorem \ref{thm:verification} provides sufficient conditions under which a smooth solution of the HJB equation \eqref{eq:HJB} coincides with the value function of the control problem \eqref{eq:value}. 
It also characterises the conditions under which an admissible control (when it exists) is optimal. The existence of such a smooth solution to the HJB equation, as well as of an optimal admissible control, can be established on a case-by-case basis depending on the application at hand. As an illustration, we rigorously and comprehensively carry out this analysis in the case study presented in Section \ref{sec:optdiv}.

Nevertheless, addressing these issues beyond a purely case-by-case framework is of independent theoretical interest; we briefly comment on this broader direction below:
\begin{enumerate}[\rm (1)]
\item  Using a penalisation approach, the early work \cite{Zhu} establishes the existence of a solution (in the almost-everywhere sense) to the dynamic programming equation associated with the considered multi-dimensional singular stochastic control problem. 
More recent contributions in this direction include \cite{Bovo-etal} and \cite{Kelbert-Moreno}, in the diffusion and jump-diffusion settings, respectively. 
The Sobolev regularity obtained in these works is typically sufficient to establish a verification theorem. Indeed, by means of a mollification procedure that allows for the application of Itô’s lemma (see, e.g.~\cite[\it Theorem 4.1, Chap.~VIII]{Fleming-Soner}), any solution to the HJB equation can thus be identified with the value function of the problem, provided that an optimal control can be constructed (see also point (3) below).

\vspace{1mm}
\item Theorem \ref{thm:verification}.(i) requires a classical solution to the HJB equation, rather than one possessing only Sobolev regularity. This requirement is justified by the fact that, in contrast to the fully multidimensional settings studied in \cite{Bovo-etal, Kelbert-Moreno, Zhu}, the HJB equation arising in our problem effectively reduces to a parameter-dependent ODE with a derivative constraint, in which the $i$-component (representing the current level of the infimum process) acts merely as a parameter. 
In such settings, $C^{2,1}$-regularity of the value function is typically observed (see, e.g.~\cite{AlMotairiZervos, GZ15, MehriZervos}). 
This observation motivates our restriction to classical $C^{2,1}$-solutions in Theorem \ref{thm:verification}.

\vspace{1mm}
\item The conditions on the admissible control $\D^{\star}$ stated in Theorem \ref{thm:verification}.(ii) require that $\D^{\star}$, together with the controlled state process $(X^{\D^{\star}}, I^{\D^{\star}})$, solve a Skorokhod reflection problem at the boundary of the so-called inaction region -- that is, the region where
\[
\big[(\mathcal{L}_{X} - r(x,i))v\big](x,i) + \Pi(x,i) = 0 \quad \text{and} \quad v_x(x,i) > f(x,i)
\]
hold true. 
Constructing a solution to the Skorokhod reflection problem is a central challenge in singular stochastic control, as it typically depends on the regularity of the free boundary separating the no-action and action regions, where the gradient constraint in the HJB equation is binding.  
A careful study of the Skorokhod reflection problem arising from \eqref{eq:value} is an important topic in its own right, and falls beyond the scope of this paper. 
For a detailed review on the construction of optimal reflecting policies in singular stochastic control, we refer the reader to the introduction of \cite{DianettiFerrari}.
\end{enumerate}
\end{remark}


\section{A problem of Optimal Dividend Payments with Multiplicative Time-Preference Impact}
\label{sec:optdiv}

In this section, we aim at applying the theoretical results developed in the previous sections to a specific dividend optimisation problem where the historical worst performance of a company's surplus process impacts the time-preference of the manager for paying dividends. 

To be more precise, we consider a company's surplus process $X^{\D}$ and its historical worst performance (modelled by its running infimum process) to be given, $\P_{x,i}$-a.s., by
\begin{align} \label{XI}
\begin{split}
X^{x,\D}_t &= x + \mu t + \eta W_t - \D_t, \quad t \geq 0, \\
I^{i,\D}_t &= i \wedge \inf_{0\leq s \leq t} X^{x,\D}_s, \qquad \quad\, t \geq 0,
\end{split}
\end{align}
where $x$ is the starting surplus value at time $0^-$, and $i$ is the company's historical worst performance in a period $(-t_0, 0)$ for some $t_0>0$, such that $i \in (-\infty,x]$ and thus $(x,i)\in \mathcal{M}$.
Here, $\D_t$ represents the cumulative amount of dividends paid from time $0^-$ up to time $t$. 
Company surplus models as in $X^{\D}$ are customary in the literature; see, e.g., \cite{AsmussenTaksar, JeanblancShyriaev, LokkaZervos, RadnerShepp}.

The decision maker aims at maximising the expected discounted dividends to be paid by controlling the dividend process $\D$, while taking into account at each time $t$ both 
the current surplus value $X^{\D}_t$, given that the company defaults when its surplus becomes zero, and
the past worst performance $I^{\D}_t$ of the company.
Clearly, a company with a relatively high value $I^{\D}$ has been performing relatively well in the past, and we assume that this raises the manager's (and investors') confidence, who thus prefers to pay dividends sooner than later.
On the contrary, the decision maker prefers to postpone dividend payments when the company's worst performance decreases and the company is thus under financial strain (such a behaviour has been observed -- and actually recommended by the ECB -- during the Covid-19 pandemic, and it is in line with the findings of \cite{DenisOsobov, FamaFrench}, among others, that observe that the propensity of paying dividends is lower in less profitable companies).

To model the aforementioned preference of the decision maker for the timing of dividend payments based on the company's historical worst performance, we assume that the impact of each unit of decrease in the company's historical worst performance $I_t$ on the decision maker's preference for delaying dividend payments is measured by a constant $q>0$ (i.e.~$q$ models the manager's sensitivity with respect to the company's performance; notice that the benchmark (and classical) model with $q=0$ will be addressed in Section \ref{sec:q=0} as the solution to this case is needed for the analysis and determination of the solution in the case $q>0$.
To be more precise, the decision maker's collective discount rate is given by 
$$
\int_{0}^{t} \big( \rho \, \d u + q \, \d I^{i,\D}_u \big), \quad t \geq 0, 
$$
where $\rho>0$ is a constant time-preference rate per unit of time (as in the classical De Finetti dividend problem), where larger values of $\rho$ indicate a preference for receiving dividends sooner than later. 
Given that the impact of the additional (performance-dependent) time-preference rate is multiplicative, the problem to be studied in this section will be called {\it optimal dividend payments with multiplicative time-preference impact}.

Overall, for any given pair of starting surplus and worst performance value $(x,i)\in \mathcal{M}$, the decision maker aims at choosing an optimal dividend policy $\D^*$ to maximise the optimisation criterion $\mathcal{J}_{x,i}$ (cf.~\eqref{eq:payoff}) with the form
\begin{equation} \label{eq:payoff-div}
\mathcal{J}_{x,i}(\D) =  \E_{x,i}\bigg[\int_0^{\tau_0^{\D}} e^{-\rho t - q I^{\D}_t} \ {\diamond}\ \d \D_t\bigg].
\end{equation}
The resulting two-dimensional singular control problem (cf.~\eqref{eq:value}) of the decision maker thus takes the form 
\begin{equation} \label{eq:value-div}
V(x,i):=\sup_{\D \in \mathcal{A}(x,i)} \E_{x,i}\bigg[\int_0^{\tau_0^{\D}} e^{-\rho t - q I^{\D}_t} \ {\diamond}\ \d \D_t\bigg],
\end{equation}
where the absorption time defined by \eqref{tau} becomes the \emph{insolvency time}
$$
\tau_0^{\D}=\inf\big\{t\geq0:\, X^{x,\D}_t \leq 0 \big\},
$$
and the admissible class of controls $\mathcal{A}(x,i)$ for the problem \eqref{eq:value-div} is defined according to \eqref{set:A} by
\begin{eqnarray*} 
&\mathcal{A}(x,i) =
\Big\{ \D \in \mathcal{S} \,:\, X^{x,\D}_{t^-} - \Delta \D_t \geq 0 \,\,\, \text{and} \,\,\, \E_{x,i}\Big[\int_0^{\tau_0^{\D}} e^{-\rho t - q I^{\D}_t} \  {\diamond}\ \d \D_t \Big] < \infty \Big\}. \nonumber
\end{eqnarray*}

Notice that, with respect to the general setup of this paper in Sections \ref{sec:problemsetting}--\ref{sec:problem}, we take
\begin{equation*}
\label{eq:ass-casestudy}
b(x,i)\equiv \mu \in \R, \,\,\, 
\sigma(x,i) \equiv \eta>0, \,\,\, 
a(i) \equiv 0, \,\,\, 
r(x,i) \equiv \rho, \,\,\, 
\Pi(x,i) \equiv 0, \,\,\, 
f(x,i)= e^{-q i}, \,\,\,
h(i)\equiv 0,
\end{equation*}
and clearly Assumptions \ref{A1} and \ref{ass:noemptyA} are satisfied (cf.\ Remark \ref{rem:Ass1}). Hence, we are in position to apply our theoretical results to solve the optimal dividend problem with multiplicative time-preference impact in \eqref{eq:value-div}.

In the sequel, we study separately the cases where the surplus process is on average decreasing or increasing in Sections \ref{sec:mu<0} and \ref{sec:mu>0}, respectively. 
In Section \ref{sec:q=0} we discuss the well-known solution to the classical optimal dividend problem, for which $q=0$. 
Before presenting the solution to the problem, it is worth noticing that the state-space of the process $(X^\D, I^\D)$, for any admissible $\D \in \mathcal{A}(x,i)$, restricts to the set (cf.~Lemma \ref{lem:state})
\begin{equation}
\label{eq:setMa-div}
\mathcal{M}^0 
= \{(x,i) \in \mathcal{M} :\, x \geq 0 \} 
= \big\{(x,i) \in \R^2_+ :\, x \geq i \geq 0 \big\} .
\end{equation}

\subsection{Benchmark case: No multiplicative time-preference impact}
\label{sec:q=0}

We first present the solution to the problem with $q=0$, which is classical, and whose solution will be needed for the analysis and construction of the solution in the case $q>0$.
In the case of $q=0$, the singular control problem \eqref{eq:value-div} takes the form of the classical one-dimensional (De Finetti) optimal dividend problem
\begin{equation} \label{eq:value-DF0}
V_0(x) :=\sup_{\D \in \tilde{\mathcal{A}}(x)} \E_{x}\bigg[\int_0^{\tau_0^{\D}} e^{-\rho t} \ \d \D_t\bigg],
\end{equation}
where $\E_{x}$ is the expectation under the measure $\P_{x}(\,\cdot\,):=\P(\,\cdot\,|X^{\D}_{0^-}=x)$ on $(\Omega,\mathcal{F})$, and 
\begin{align} \label{set:Ax}
&\tilde{\mathcal{A}}(x) =
\Big\{ \D \in \mathcal{S} \,:\, X^{x,\D}_{t^-} - \Delta \D_t \geq 0 
\quad \text{and} \quad \E_{x}\Big[\int_0^{\tau_0^{\D}} e^{-\rho t} \ \d \D_t \Big] < \infty \Big\}. 
\end{align}
This problem is nowadays well-understood and its solution can be found in \cite{JeanblancShyriaev, LokkaZervos}, among others.

On the one hand, if $\mu\leq0$, the company's surplus process is decreasing on average. 
In this case, it turns out that it is optimal to immediately payout all the surplus as dividends~and~we~have
\begin{equation} \label{eq:V0muneg}
V_0(x) = x \quad \text{and} \quad \widehat{D}_t=x,\,\, t\geq0, \quad \widehat{D}_{0^-}=0. 
\end{equation}

On the other hand, if $\mu>0$, the surplus process is (in absence of dividend payments) increasing on average. 
It turns out that in this case the optimal dividend policy in \eqref{eq:value-DF0} is a so-called `threshold-strategy', that prescribes paying dividends either in a lump-sum when $x>b_\circ$, or in a Skorokhod-reflection type by reflecting the surplus process $X^\D$ when reaching the boundary $b_\circ$ from below. 
Hence, letting $\alpha>0$ and $\beta<0$ be the roots of the characteristic equation $\frac{1}{2}\sigma^2 \theta^2 + \mu \theta - \rho=0$, the value function when $\mu>0$ is given by 
\begin{equation} \label{eq:sol-DF0}
V_0(x) = \begin{cases}
- \frac{\beta}{\alpha(\alpha - \beta)} e^{\alpha (x - b_o)} + \frac{\alpha}{\beta(\alpha - \beta)}e^{\beta (x - b_o)}, 
\quad &0 < x < b_\circ, \\
x - b_\circ + \frac{\mu}{\rho}, &x \geq b_\circ,
\end{cases}
\end{equation}
where the optimal threshold $b_\circ$ is defined by 
\begin{equation} \label{eq:bo}
b_\circ := \frac{1}{\alpha - \beta}\ln\Big(\frac{\beta^2}{\alpha^2}\Big). 
\end{equation}
Furthermore, the optimal dividend policy is given by
\begin{equation} \label{eq:OCq0}
\widehat{\D}_t:=\sup_{0\leq s \leq t}\big(x + \mu s + \sigma W_s - b_o\big)^+,\,\, t\geq 0, \quad \widehat{\D}_{0^-}:=0.
\end{equation}

As in the classical optimal dividend problem reviewed in this section, we distinguish the analysis of the two-dimensional singular control problem \eqref{eq:value-div} in the forthcoming sections into the cases of $\mu\leq0$ and $\mu>0$.

\subsection{Dividend problem with multiplicative time-preference impact: Case of $\mu\leq 0$.}
\label{sec:mu<0}

We first show that, as in the classical De Finetti's dividend problem, when the average rate of increase is $\mu\leq 0$, i.e.~the surplus is on average non-increasing, then the policy ``Pay all surplus as dividends immediately" is optimal (see Figure \ref{Fig0}), even in the presence of a multiplicative time-preference impact with rate $q>0$ on the decision maker.

\begin{proposition}
\label{prop:muneg}
Suppose $\mu \leq 0$ in \eqref{XI}. Then the control 
\begin{equation}
\label{eq:optimalcontrolmuneg}
\D^{\star}_t= x, \quad t \geq 0, \quad \D^{\star}_{0^-}=0,
\end{equation}
is optimal and the value function in \eqref{eq:value-div} takes the form
\begin{equation}
\label{eq:valuefctmuneg}
V(x,i)=e^{-q i} \Big(x - i - \frac{1}{q} \Big) + \frac{1}{q}, \quad (x,i) \in \mathcal{M}^0 .
\end{equation}
\end{proposition}

\begin{proof}
We firstly notice by the definition \eqref{eq:optimalcontrolmuneg} of the control process $\D^{\star}$, the insolvency time becomes $\tau_0^{\D^{\star}} = 0$.
Then, recalling \eqref{defintegral} (cf.\ Definition \ref{int:D}), the facts that $f(x,i) = e^{-q i}$ and $(X^{\D^{\star}}_{0^-}, I^{\D^{\star}}_{0^-})=(x,i)$, 
the payoff $v$ associated to the control \eqref{eq:optimalcontrolmuneg} is given by 
\begin{align} \label{cand}
\begin{split}
v(x,i)
&= \E_{x,i}\bigg[\int_0^{\tau_0^{\D^{\star}}} e^{-\rho t - q I^{\D^{\star}}_t} {\diamond}\ \d \D^{\star}_t\bigg] \\
&= \int_0^{x-i} e^{-q i} \, \d u 
+ \int_{x-i}^x e^{-q(x-u)} \, \d u \\
&= e^{-q i} (x-i) + \frac{1}{q} (1-e^{-qi})
= e^{-q i} \Big(x - i - \frac{1}{q} \Big) + \frac{1}{q}. 
\end{split}
\end{align}

We firstly notice that, since $v$ is the payoff associated to the admissible control $D^{\star} \in \mathcal{A}(x,i)$, we have by construction and the definition \eqref{eq:value-div} of $V$ that $v\leq V$.

It thus remains to show that $v\geq V$ and for this we verify, in the remainder of this proof, the requirements of the statement of Theorem \ref{thm:verification}.(i). 

The candidate $v$ from \eqref{cand} clearly belongs to $C^{2,1}(\mathcal{M}^0)$ and it is such that
\begin{align*}
&v(0,0) = 0, 
\qquad v_x(x,i) = e^{-q i}, 
\qquad v_i(i,i)= 0, \\
&\quad \mathcal{L}_X v(x,i) - \rho v(x,i) = \mu e^{-q i} - \rho v(x,i) \leq 0.
\end{align*}
It can thus identify with a solution to \eqref{eq:HJB} satisfying \eqref{eq:bdconds1} and \eqref{eq:bdconds2}.

Then, given that $\P_{x,i}(\tau_0^{\D^\star}=+\infty)=0$, it only remains to prove \eqref{eq:integrability}. 
To that end, we notice that $0\leq v(x,i) \leq C(1+x)$, for a suitable constant $C>0$, and that $X^D_t \leq X^0_t$ for all $t\geq 0$, $\P$-a.s.. 
Therefore, \eqref{eq:integrability} readily follows thanks to  
$\E[\sup_{t\geq0} e^{-\rho t} |X^0_t|] < \infty,$
and completes~the~proof.
\end{proof}

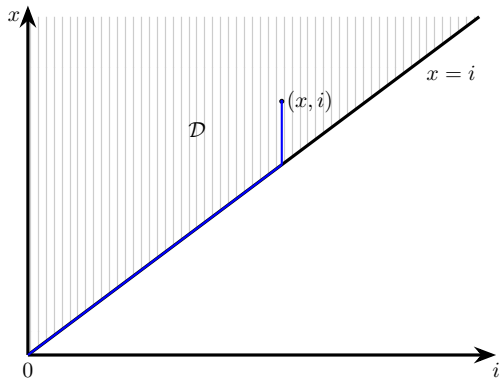
\begin{figure}
\centering
\begin{tikzpicture}
\begin{scope}[scale=0.75, transform shape]
\coordinate[label = below:$0$] (A) at (0,0){};
\coordinate[label = below:$i$] (D) at (8.3,0){};
\coordinate[label = left:$x$] (E) at (0,6){};
\coordinate (F) at (8,1){};
\coordinate (G) at (4,3){};
\coordinate (I) at (4,6){};

\fill[pattern=vertical lines, pattern color=gray!40] (A) -- (E) -- (8,6) -- cycle; 
\node at (3,4) {$\mathcal{D}$};

\draw[very thick, -{Stealth}] (A) -- (8.3,0);
\draw[very thick, -{Stealth}] (A) -- (0,6.2);
\draw[very thick] (A) -- (8,6);
\node at (7.5,5) {$x=i$};

\coordinate (X) at (4.5,4.5){};
\draw[fill=blue] (X) circle (1pt);
\node at (5,4.5) {$(x,i)$};

\draw[thick, blue] (X) -- (4.5,3.375);
\draw[thick, blue] (4.5,3.375) -- (A);

\end{scope}
\end{tikzpicture}
\captionsetup{width=1\textwidth}
\caption{\footnotesize Schematic depiction of the state-space $\mathcal{M}^0 = \mathcal{D}$ given by \eqref{eq:setMa-div} and identifying with the action region $\mathcal{D}$ for the problem \eqref{eq:value-div} with $\mu \leq 0$, 
and the movement (blue) of the process $(X^{x,\rm D},I^{x,\rm D})$ started from $(x,i) \in \mathcal{D}$ such that $x>i$, which jumps downwards in a `hockey-stick' direction to the origin, where it is absorbed.}
\label{Fig0}
\end{figure}

\subsection{Dividend problem with multiplicative time-preference impact: Case of $\mu > 0$.}
\label{sec:mu>0}

In this section we consider the more interesting case (compared to the solution in Section \ref{sec:mu<0}), when the average rate of increase (in the absence of dividend payments) is $\mu>0$. 

In this case, we conjecture that there exists a free boundary function 
$b:\mathcal{O} \to \R$ 
that separates the waiting region $\mathcal{C}$ from the action region $\mathcal{D}$, such that 
\begin{align} \label{CD}
\begin{split}
\mathcal{C} 
&= \{(x,i)\in \mathcal{M}^0 \,:\, 0 < i \leq x < b(i) \}, \\
\mathcal{D} &= \{(x,i)\in \mathcal{M}^0 \,:\, x \geq b(i) \vee i \} .
\end{split}
\end{align}
In other words, when the surplus process is `high enough', i.e.~above a threshold depending on the company's worst performance then the company pays dividends. 

In view of Theorem \ref{thm:verification}, we look for a smooth solution to the HJB equation \eqref{eq:HJB} subject to the boundary conditions \eqref{eq:bdconds1} and \eqref{eq:bdconds2}. 
With this in mind, and the structure of the waiting $\mathcal{C}$ and action $\mathcal{D}$ regions in \eqref{CD}, we aim at determining the pair of functions $(v,b)$ such that the following free-boundary problem is satisfied in the classical sense:
\begin{align}
& \tfrac{1}{2}\sigma^2 v_{xx}(x,i) + \mu v_x(x,i) - \rho v = 0, \quad 0\leq i < x < b(i), \label{eq:FBP1} \\
& v_x(x,i) = e^{-q i}, \quad \quad \quad \quad \quad \quad \quad \quad \quad 0\leq i < x \quad \text{and} \quad x \geq b(i),  \label{eq:FBP2} \\
& \tfrac{1}{2}\sigma^2 v_{xx}(x,i) + \mu v_x(x,i) - \rho v \leq 0, \quad  0\leq i < x, \label{eq:FBP3}  \\
& v_x(x,i) \geq e^{-q i}, \quad \quad \quad \quad \quad \quad  \quad \quad \quad 0\leq i < x, \label{eq:FBP4}  \\
& v_i(i,i)=0, \quad \quad \quad \quad \quad \quad \quad \quad \quad \quad \;\;\, i>0, \label{eq:FBP5}  \\
& v(0,0) = 0. \label{eq:FBP6}
\end{align}
Here, the equations \eqref{eq:FBP1}--\eqref{eq:FBP2} and the inequalities \eqref{eq:FBP3}--\eqref{eq:FBP4} are a consequence of the HJB equation \eqref{eq:HJB}, the equation \eqref{eq:FBP5} comes from the Neumann condition \eqref{eq:bdconds1}, and the equation \eqref{eq:FBP6} comes from the Dirichlet boundary condition \eqref{eq:bdconds2}. 

In order to solve this free-boundary problem, we firstly notice that the ODE \eqref{eq:FBP1} is solved by a function 
\begin{equation} \label{eq:solODE}
v(x,i) = A(i) e^{\alpha x} + B(i) e^{\beta x}, 
\quad \text{for all} \quad  0 \leq i < x < b(i),
\end{equation}
where 
$$\alpha:=\frac{-\mu - \sqrt{\mu^2 + 2\rho\sigma^2}}{\sigma^2} < 0 
\quad \text{and} \quad 
\beta:=\frac{-\mu + \sqrt{\mu^2 + 2\rho\sigma^2}}{\sigma^2} > 0$$ are the solutions to
$$\frac{1}{2}\sigma^2 \theta^2 + \mu \theta - \rho =0,$$
and the functions $A$ and $B$ are to be determined. 

By applying the condition \eqref{eq:FBP6} to the expression \eqref{eq:solODE} of $v$, we then find that  
\begin{equation}
\label{eq:A0B0}
A(0) + B(0) = 0,
\end{equation}
while using \eqref{eq:FBP5} yields
\begin{equation}
\label{eq:FBP5bis}
A'(i) e^{\alpha i} + B'(i) e^{\beta i} = 0, \quad \text{for all} \quad i>0.
\end{equation}

Now, in order to ensure smoothness of the candidate value function (cf.\ Theorem \ref{thm:verification}, $v$ must be in $C^{2,1}(\mathcal{M}^0)$), we also require that $v(\cdot,i)$ is twice-continuously differentiable across the free boundary $b(i)$. 
By differentiating the expression \eqref{eq:solODE} with respect to $x$ and combining it with the condition \eqref{eq:FBP2}, we have for all $i \geq 0$ that  
\begin{equation}
\label{eq:smothfit1}
v_x(b(i)-,i) = v_x(b(i)+,i)
\quad \Leftrightarrow \quad  \alpha 
A(i) e^{\alpha b(i)} + \beta B(i) e^{\beta b(i)} = e^{-q i},
\end{equation}
and
\begin{equation}
\label{eq:smothfit2}
v_{xx}(b(i)-,i) = v_{xx}(b(i)+,i)
\quad \Leftrightarrow \quad 
\alpha^2 A(i) e^{\alpha b(i)} + \beta^2 B(i) e^{\beta b(i)} = 0.
\end{equation}
Solving \eqref{eq:smothfit1} and \eqref{eq:smothfit2} with respect to $A(i)$ and $B(i)$ we find that
\begin{equation}
\label{eq:AiBi}
A(i)= - \frac{\beta}{\alpha(\alpha - \beta)}e^{-q i - \alpha b(i)} \quad \text{and} \quad B(i)= \frac{\alpha}{\beta(\alpha - \beta)}e^{-q i - \beta b(i)}, 
\quad \text{for all} \quad i \geq 0.
\end{equation}

Hence, by substituting the expressions \eqref{eq:AiBi} of $A$ and $B$ back to the equation \eqref{eq:solODE} of $v$, we obtain 
\begin{equation} \label{eq:vAB}
v(x,i) = - \frac{\beta}{\alpha(\alpha - \beta)} e^{-q i + \alpha (x - b(i))} + \frac{\alpha}{\beta(\alpha - \beta)}e^{-q i + \beta (x - b(i))}, 
\quad \text{for all} \quad  0 \leq i < x < b(i).
\end{equation}
Furthermore, by evaluating the expressions \eqref{eq:AiBi} of $A$ and $B$ for $i=0$ and substituting the resulting expressions back to the condition \eqref{eq:A0B0}, we obtain that (cf.\ \eqref{eq:bo})
\begin{equation} \label{eq:b0}
    b(0) = \frac{1}{\alpha - \beta}\ln\Big(\frac{\beta^2}{\alpha^2}\Big) =  b_\circ > 0, 
\end{equation}
where the latter equality follows from the definition \eqref{eq:bo} of the threshold $b_o$, which defines the optimal dividend policy in the case of no multiplicative time-preference impact (cf.~Section \ref{sec:q=0}).
The resulting positivity follows from the fact that $\alpha<0<\beta$ and $|\alpha|>\beta$.

A straightforward differentiation of the expressions \eqref{eq:AiBi} of $A$ and $B$ and their substitution into the ODE~\eqref{eq:FBP5bis}, then yields the Cauchy problem equation (ODE) for $b$:
\begin{align} \label{eq:ODEb}
\begin{cases}
\begin{split}
b'(i) 
&= \mathcal{B}(b(i),i), \quad i>0, \\
b(0) &= b_\circ, 
\end{split}
\end{cases}
\end{align}
where 
\begin{equation} \label{flow}
\mathcal{B}(b,i) = \frac{q}{\alpha\beta} \bigg( \frac{\alpha^2e^{\beta(i-b)} - \beta^2 e^{\alpha(i-b)}}{\beta e^{\alpha(i-b)} - \alpha e^{\beta(i-b)}}\bigg) .
\end{equation}
We aim at solving this ODE in the domain 
\begin{equation} \label{Db}
D_b := \{(b,i) \in \R^2 \,:\, i > 0 \} .
\end{equation}

\begin{theorem} \label{thm:ODEb}
There exists a unique solution to the Cauchy problem \eqref{eq:ODEb}--\eqref{flow}. 
This solution, still denoted by $b$, is is strictly decreasing on $(0,\infty)$ such that 
$$
\lim_{i \downarrow 0} b'(i)=0
\quad \text{and} \quad
\lim_{i \uparrow \infty} b(i)= -\infty .
$$
\end{theorem}
\begin{proof}
We prove the various parts of the statement separately. 
\vspace{0.15cm}

\emph{Existence and Uniqueness.} 
Given that $\mathcal{B}(b_\circ,0)$ is well-defined (equal to zero thanks to \eqref{limb}), we can simply solve the ODE 
$$ 
b'(i) = \mathcal{B}(b(i),i), \quad \text{for $i>0$},
$$
with initial condition $b(0) = b_\circ$. 
By observing that $\mathcal{B}$ in \eqref{flow} is locally Lipschitz continuous in the open domain $D_{b}$ from \eqref{Db}, it follows that for  there exist a unique strictly decreasing function 
$b(\cdot) : (0,\infty) \to D_b$ that satisfies the ODE with initial condition as in \eqref{eq:ODEb} (see \cite[Theorems I.1.4 and I.1.5]{Piccinini}, or \cite[Theorems 1.1 and 3.1]{Hartman}). 

\vspace{0.15cm}
\emph{Limit properties as $i \downarrow 0$.} 
From the ODE in \eqref{eq:ODEb}--\eqref{flow} it is readily seen that 
\begin{equation} \label{limb}
\lim_{i \downarrow 0} b'(i) 
= \mathcal{B}(b(0),0) 
= \frac{q}{\alpha\beta} \bigg( \frac{\alpha^2e^{- \beta b(0)} - \beta^2 e^{- \alpha b(0)}}{\beta e^{- \alpha b(0)} - \alpha e^{- \beta b(0)}}\bigg) = 0,
\end{equation}
upon recalling the boundary condition $b(0)=b_\circ$ and its expression in \eqref{eq:b0}. 

\vspace{0.15cm}
\emph{Monotonicity.} 
Given that $\alpha<0$ and $\beta>0$, we readily see that the denominator in \eqref{flow} is always positive, hence simple calculations reveal that 
\begin{equation} \label{b'<0}
b'(i) < 0 , \quad i>0
\quad \Leftrightarrow \quad 
b(i) < i + b_\circ, \quad i>0.  
\end{equation}

A combination of the limits $\lim_{i \downarrow 0} b(i)=b_\circ$ and $\lim_{i \downarrow 0} b'(i)=0$ implies that there exists a sufficiently small $\varepsilon>0$, such that $b(i)< i + b_\circ$, for all $i \in (0,\varepsilon)$. 
This further implies thanks to \eqref{b'<0} that $b'(i)<0$, for all $i \in (0,\varepsilon)$, and it is therefore straightforward to see that \eqref{b'<0} will remain true for all $i \geq \varepsilon$ as well. 
This concludes that $i\mapsto b(i)$ is strictly decreasing on $(0,\infty)$. 
\vspace{0.15cm}

\emph{Limit properties as $i \uparrow \infty$.}
The latter solution further satisfies 
$$
\lim_{i \uparrow \infty} b(i) = - \infty .
$$
To see this, assume (aiming for a contradiction) that there exists a finite limit $b_\infty := \lim_{i \uparrow \infty}b(i) \in (-\infty,\infty)$, which implies that $\lim_{i \uparrow \infty}b'(i)=0$. 
However, we observe from \eqref{flow} and straightforward calculations that 
$$
\lim_{i \uparrow \infty}b'(i) 
= \lim_{i \uparrow \infty} \mathcal{B}(b(i),i)  
= - \frac{q}{\beta} < 0,
$$
which provides the desired contradiction and completes the proof. 
\end{proof}

We then define the function $g:[0,\infty) \to \R$, by 
\begin{equation} \label{eq:g}
g(i) := b(i)-i, \quad i \geq 0.
\end{equation}
The usefulness of this function is that its zeros determine the instances that the boundary function $b$ touches the diagonal 
$$
\partial \mathcal{M}^0 = \big\{ (x,i) \in \R^2 :\, x = i \geq 0 \big\}
$$
of the state-space $\mathcal{M}^0$ given by \eqref{eq:setMa-div}.
These zeros are important in the forthcoming analysis, especially in the split of the action region $\mathcal{D}$ into subregions where the control process $\D$ induces different behaviour for the two-dimensional process $(X^\D,I^\D)$; see Section \ref{sec:v} for further details.
In the following lemma we prove that there is only one such zero in this control problem. 
 
\begin{lemma} \label{lem:i^*}
There exists a unique critical minimum surplus value $i^\star$ solving the equation $g(i)=0$, where $g$ is defined by \eqref{eq:g}, which satisfies $i^\star \in (0,b_\circ)$ and further implies that 
$$
b(i) \begin{cases}
> i, & \text{for } i < i^\star, \\
= i, & \text{for } i = i^\star, \\
< i, & \text{for } i > i^\star.
\end{cases}
$$
\end{lemma}
\begin{proof}
Due to the decreasing property of $b(\cdot)$ in Theorem \ref{thm:ODEb}, we notice that $g$ is strictly decreasing on the whole $[0,\infty)$.

Then, we observe that 
$$
g(0)=b_\circ>0 
\quad \text{and} \quad 
g(b_\circ) = b(b_\circ) - b_\circ <0,
$$
thanks to \eqref{eq:ODEb}, \eqref{eq:b0}, and again the decreasing property of $b(\cdot)$. 

Hence, there exists a unique root of $g$ on $[0,\infty)$, denoted by $i^{\star}$, which satisfies all desired properties in the statement of the lemma and thus completes the proof.
\end{proof}

In light of Lemma \ref{lem:i^*}, we conclude that the boundary $b(\cdot)$ is part of the state-space $\mathcal{M}^0$ only on $[0,i^\star]$. This is an important observation that plays a pivotal role in the forthcoming construction of the value function. 

\subsubsection{Construction of the candidate value function}
\label{sec:v}

In light of the conjectured structure of the state-space $\mathcal{M}^0 = \mathcal{C} \cup \mathcal{D}$ in \eqref{CD} and the results in Lemma \ref{lem:i^*} for the boundary function $b$ defining the regions $\mathcal{C}$ and $\mathcal{D}$, we split the action region $\mathcal{D}$ from \eqref{CD} into two distinct components (cf.\ the critical minimum surplus value $i^\star$ in Lemma \ref{lem:i^*}), given by the subsets
\begin{align} \label{eq:stopping}
\begin{split}
\mathcal{D}_1 &:= \{(x,i)\in \mathcal{M}^0 \,:\, x \geq b(i) > i \,,\; i < i^\star \}, \\
\mathcal{D}_2 &:= \{(x,i)\in \mathcal{M}^0 \,:\, x \geq i \geq b(i) \,,\; i \geq i^\star \},
\end{split}
\end{align}
such that $\mathcal{D} = \mathcal{D}_1 \cup \mathcal{D}_2$ and $\mathcal{D}_1 \cap \mathcal{D}_2 = \emptyset$ (see Figure \ref{Fig2}).
We are now ready to construct a candidate value function. 
To that end, we proceed with studying each component of the state-space $\mathcal{C}$, $\mathcal{D}_1$ and $\mathcal{D}_2$ separately in what follows.

The waiting region $\mathcal{C}$ defined by the associated set in \eqref{CD} is the part of the state-space where it should be optimal not to pay dividends, but the process $(X^{x,\D}, I^{i,\D})$ is reflected in the south-west direction along the diagonal, when $X^{x,\D}=I^{i,\D}$ and $I^{i,\D}$ moves to a new infimum value 
-- this has the effect of changing the value of $b(I^{i,\D})$ as well -- 
until the company's default when $X^{x,\D}=I^{i,\D}=0$. 
In the meantime, upon exiting the waiting region $\mathcal{C}$, i.e.~every time $X^{x,\D} = b(I^{i,\D})$, the manager should pay dividends continuously in a manner such that the process $X^{x,\D}$ is reflected downwards on $b(I^{i,\D})$ in an appropriate way (see Figure \ref{Fig2}, and Section \ref{sec:control} for the rigorous construction of the dividend policy $\D$).
In the waiting region $\mathcal{C}$, we thus take the previously constructed smooth solution to \eqref{eq:FBP1} satisfying 
the reflection condition \eqref{eq:FBP2} at the boundary $b$, 
the reflection condition \eqref{eq:FBP5} along the diagonal, 
and the absorption condition \eqref{eq:FBP6} at the origin, 
which was obtained in \eqref{eq:vAB}. 
We therefore have 
\begin{equation} \label{eq:vonC}
v(x,i) = \frac{e^{-q i}}{\alpha-\beta} \Big(\frac{\alpha}{\beta}e^{\beta(x-b(i))} - \frac{\beta}{\alpha}e^{\alpha(x-b(i))}\Big), 
\quad (x,i)\in \mathcal{C}.
\end{equation}

\begin{figure}
\centering
\begin{tikzpicture}
\begin{scope}[scale=0.75, transform shape]
\coordinate[label = below:$0$] (A) at (0,0){};
\coordinate[label = left:$i^\star$] (B) at (0,3){};
\coordinate[label = below:$i^\star$] (C) at (4,0){};
\coordinate[label = below:$i$] (D) at (8.3,0){};
\coordinate[label = left:$x$] (E) at (0,6){};
\coordinate (F) at (8,1){};
\coordinate (G) at (4,3){};
\coordinate[label=left:$b_\circ$] (H) at (0,4.5){};
\coordinate (I) at (4,6){};

\node at (1.25,2) {$\mathcal{C}$};

\fill[pattern=north east lines, pattern color=gray!40] (E) -- (H) -- (H) .. controls +(1,0) and +(-1,0.2) .. (G) -- (I) -- cycle; 
\node at (2,5) {$\mathcal{D}_1$};

\fill[pattern=north west lines, pattern color=gray!40] (I) -- (G) -- (8,6) -- cycle; 
\node at (5,5) {$\mathcal{D}_2$};

\draw[very thick, -{Stealth}] (A) -- (8.3,0);
\draw[very thick, -{Stealth}] (A) -- (0,6.2);
\draw[very thick] (A) -- (8,6);
\node at (7.5,5) {$x=i$};

\draw[dashed, thick] (B) -- (G);
\draw[dashed, thick] (C) -- (I);

\coordinate (X) at (2,2.5){};
\draw[fill=blue] (X) circle (1pt);
\node at (2.5,2.5) {$(x,i)$};

\draw[thick, blue] (X) -- (2,3.5);
\draw[thick, blue] (X) -- (2,1.5);
\draw[thick, blue] (2,1.5) -- (A);
\draw[thick, blue] (1.8,1.35) -- (1.8,3.4);
\draw[thick, blue] (1.5,1.125) -- (1.5,4.05);
\draw[thick, blue] (0.9,0.675) -- (0.9,4.3);
\draw[thick, blue] (0.8,0.6) -- (0.8,3.6);
\draw[thick, blue] (0.6,0.45) -- (0.6,4.4);
\draw[thick, blue] (0.2,0.15) -- (0.2,2.15);

\draw[thick, red] (H) .. controls +(1,0) and +(-1,0.2) .. (G);
\draw[thick, red] (G) .. controls +(1,-0.2) and +(-1,1) .. (8,1);

\node at (7,1) {$x=b(i)$};

\draw[fill=black] (G) circle (1pt);
\draw[fill=black] (H) circle (1pt);

\end{scope}
\end{tikzpicture}
\begin{tikzpicture}
\begin{scope}[scale=0.75, transform shape]
\coordinate[label = below:$0$] (A) at (0,0){};
\coordinate[label = left:$i^\star$] (B) at (0,3){};
\coordinate[label = below:$i^\star$] (C) at (4,0){};
\coordinate[label = below:$i$] (D) at (8.3,0){};
\coordinate[label = left:$x$] (E) at (0,6){};
\coordinate (F) at (8,1){};
\coordinate (G) at (4,3){};
\coordinate[label=left:$b_\circ$] (H) at (0,4.5){};
\coordinate (I) at (4,6){};

\node at (1.25,2) {$\mathcal{C}$};

\fill[pattern=north east lines, pattern color=gray!40] (E) -- (H) -- (H) .. controls +(1,0) and +(-1,0.2) .. (G) -- (I) -- cycle; 
\node at (2,5) {$\mathcal{D}_1$};

\fill[pattern=north west lines, pattern color=gray!40] (I) -- (G) -- (8,6) -- cycle; 
\node at (5,5) {$\mathcal{D}_2$};

\draw[very thick, -{Stealth}] (A) -- (8.3,0);
\draw[very thick, -{Stealth}] (A) -- (0,6.2);
\draw[very thick] (A) -- (8,6);
\node at (7.5,5) {$x=i$};

\draw[dashed, thick] (B) -- (G);
\draw[dashed, thick] (C) -- (I);

\coordinate (X) at (2.5,5){};
\draw[fill=blue] (X) circle (1pt);
\node at (3,5) {$(x,i)$};

\draw[thick, blue] (X) -- (2.5,1.875);
\draw[thick, blue] (2.5,1.875) -- (A);
\draw[thick, blue] (2.4,1.8) -- (2.4,3.63);
\draw[thick, blue] (2.2,1.65) -- (2.2,3.73);
\draw[thick, blue] (2.05,1.5375) -- (2.05,3.8);
\draw[thick, blue] (1.8,1.35) -- (1.8,3.4);
\draw[thick, blue] (1.5,1.125) -- (1.5,4.05);
\draw[thick, blue] (0.9,0.675) -- (0.9,4.3);
\draw[thick, blue] (0.8,0.6) -- (0.8,3.6);
\draw[thick, blue] (0.6,0.45) -- (0.6,4.4);
\draw[thick, blue] (0.2,0.15) -- (0.2,2.15);

\draw[thick, red] (H) .. controls +(1,0) and +(-1,0.2) .. (G);
\draw[thick, red] (G) .. controls +(1,-0.2) and +(-1,1) .. (8,1);

\node at (7,1) {$x=b(i)$};

\draw[fill=black] (G) circle (1pt);
\draw[fill=black] (H) circle (1pt);
\end{scope}
\end{tikzpicture}
\begin{tikzpicture}
\begin{scope}[scale=0.75, transform shape]
\coordinate[label = below:$0$] (A) at (0,0){};
\coordinate[label = left:$i^\star$] (B) at (0,3){};
\coordinate[label = below:$i^\star$] (C) at (4,0){};
\coordinate[label = below:$i$] (D) at (8.3,0){};
\coordinate[label = left:$x$] (E) at (0,6){};
\coordinate (F) at (8,1){};
\coordinate (G) at (4,3){};
\coordinate[label=left:$b_\circ$] (H) at (0,4.5){};
\coordinate (I) at (4,6){};

\node at (1.25,2) {$\mathcal{C}$};

\fill[pattern=north east lines, pattern color=gray!40] (E) -- (H) -- (H) .. controls +(1,0) and +(-1,0.2) .. (G) -- (I) -- cycle; 
\node at (2,5) {$\mathcal{D}_1$};

\fill[pattern=north west lines, pattern color=gray!40] (I) -- (G) -- (8,6) -- cycle; 
\node at (5,5) {$\mathcal{D}_2$};

\draw[very thick, -{Stealth}] (A) -- (8.3,0);
\draw[very thick, -{Stealth}] (A) -- (0,6.2);
\draw[very thick] (A) -- (8,6);
\node at (7.5,5) {$x=i$};

\draw[dashed, thick] (B) -- (G);
\draw[dashed, thick] (C) -- (I);

\coordinate (X) at (6,4.5){};
\draw[fill=blue] (X) circle (1pt);
\node at (6.35,4.25) {$(x,i)$};

\draw[thick, blue] (X) -- (4.5,3.375);
\draw[thick, blue] (4.5,3.375) -- (A);
\draw[thick, blue] (3.8,2.85) -- (3.8,3.05);
\draw[thick, blue] (3.5,2.625) -- (3.5,3.16);
\draw[thick, blue] (3.1,2.325) -- (3.1,3.32);
\draw[thick, blue] (3,2.25) -- (3,3.1);
\draw[thick, blue] (2.6,1.95) -- (2.6,3.2);
\draw[thick, blue] (2.4,1.8) -- (2.4,3.4);
\draw[thick, blue] (2.2,1.65) -- (2.2,3.73);
\draw[thick, blue] (2.05,1.5375) -- (2.05,3.8);
\draw[thick, blue] (1.8,1.35) -- (1.8,3.4);
\draw[thick, blue] (1.5,1.125) -- (1.5,4.05);
\draw[thick, blue] (0.9,0.675) -- (0.9,4.3);
\draw[thick, blue] (0.8,0.6) -- (0.8,3.6);
\draw[thick, blue] (0.6,0.45) -- (0.6,4.4);
\draw[thick, blue] (0.2,0.15) -- (0.2,2.15);

\draw[thick, red] (H) .. controls +(1,0) and +(-1,0.2) .. (G);
\draw[thick, red] (G) .. controls +(1,-0.2) and +(-1,1) .. (8,1);

\node at (7,1) {$x=b(i)$};

\draw[fill=black] (G) circle (1pt);
\draw[fill=black] (H) circle (1pt);

\end{scope}
\end{tikzpicture}
\begin{tikzpicture}
\begin{scope}[scale=0.75, transform shape]
\coordinate[label = below:$0$] (A) at (0,0){};
\coordinate[label = left:$i^\star$] (B) at (0,3){};
\coordinate[label = below:$i^\star$] (C) at (4,0){};
\coordinate[label = below:$i$] (D) at (8.3,0){};
\coordinate[label = left:$x$] (E) at (0,6){};
\coordinate (F) at (8,1){};
\coordinate (G) at (4,3){};
\coordinate[label=left:$b_\circ$] (H) at (0,4.5){};
\coordinate (I) at (4,6){};

\node at (1.25,2) {$\mathcal{C}$};

\fill[pattern=north east lines, pattern color=gray!40] (E) -- (H) -- (H) .. controls +(1,0) and +(-1,0.2) .. (G) -- (I) -- cycle; 
\node at (2,5) {$\mathcal{D}_1$};

\fill[pattern=north west lines, pattern color=gray!40] (I) -- (G) -- (8,6) -- cycle; 
\node at (5,5) {$\mathcal{D}_2$};

\draw[very thick, -{Stealth}] (A) -- (8.3,0);
\draw[very thick, -{Stealth}] (A) -- (0,6.2);
\draw[very thick] (A) -- (8,6);
\node at (7.5,5) {$x=i$};

\draw[dashed, thick] (B) -- (G);
\draw[dashed, thick] (C) -- (I);

\coordinate (X) at (4.5,4.5){};
\draw[fill=blue] (X) circle (1pt);
\node at (5,4.5) {$(x,i)$};

\draw[thick, blue] (X) -- (4.5,3.375);
\draw[thick, blue] (4.5,3.375) -- (A);
\draw[thick, blue] (3.8,2.85) -- (3.8,3.05);
\draw[thick, blue] (3.5,2.625) -- (3.5,3.16);
\draw[thick, blue] (3.1,2.325) -- (3.1,3.32);
\draw[thick, blue] (3,2.25) -- (3,3.1);
\draw[thick, blue] (2.6,1.95) -- (2.6,3.2);
\draw[thick, blue] (2.4,1.8) -- (2.4,3.4);
\draw[thick, blue] (2.2,1.65) -- (2.2,3.73);
\draw[thick, blue] (2.05,1.5375) -- (2.05,3.8);
\draw[thick, blue] (1.8,1.35) -- (1.8,3.4);
\draw[thick, blue] (1.5,1.125) -- (1.5,4.05);
\draw[thick, blue] (0.9,0.675) -- (0.9,4.3);
\draw[thick, blue] (0.8,0.6) -- (0.8,3.6);
\draw[thick, blue] (0.6,0.45) -- (0.6,4.4);
\draw[thick, blue] (0.2,0.15) -- (0.2,2.15);

\draw[thick, red] (H) .. controls +(1,0) and +(-1,0.2) .. (G);
\draw[thick, red] (G) .. controls +(1,-0.2) and +(-1,1) .. (8,1);

\node at (7,1) {$x=b(i)$};

\draw[fill=black] (G) circle (1pt);
\draw[fill=black] (H) circle (1pt);

\end{scope}
\end{tikzpicture}
\captionsetup{width=1\textwidth}
\caption{\footnotesize Schematic depiction of the movement (blue) of the process $(X^{x,\rm D},I^{x,\rm D})$ started from $(x,i) \in \mathcal{M}^0 = \mathcal{C} \cup \mathcal{D}_1 \cup \mathcal{D}_2$ given by \eqref{eq:setMa-div}, \eqref{CD} and \eqref{eq:stopping} for the problem \eqref{eq:value-div} with $\mu>0$. 
Top left panel: The process $(X^{x,\rm D},I^{x,\rm D})$ started from $(x,i)\in \mathcal{C}$ is reflected either downwards at the boundary function $x = b(i)$ (red) or in the south-west direction at the diagonal $x=i$ (black), until it is absorbed at the origin.
Top right panel: The process $(X^{x,\rm D},I^{x,\rm D})$ started from $(x,i)\in \mathcal{D}_1$ jumps downwards to $(b(i),i)$ at time $0$ and then continues as in top left panel.
Bottom left panel: The process $(X^{x,\rm D},I^{x,\rm D})$ started from $(x,i)\in \mathcal{D}_2$ jumps downwards in the south-west direction to $(b(i^\star),i^\star) = (i^\star,i^\star)$ at time $0$ and then continues as in top left panel.
Bottom right panel: The process $(X^{x,\rm D},I^{x,\rm D})$ started from $(x,i)\in \mathcal{D}_2$ jumps downwards in a `hockey-stick' direction to $(b(i^\star),i^\star)=(i^\star,i^\star)$ at time $0$ and then continues as in top left panel.}
\label{Fig2}
\end{figure}
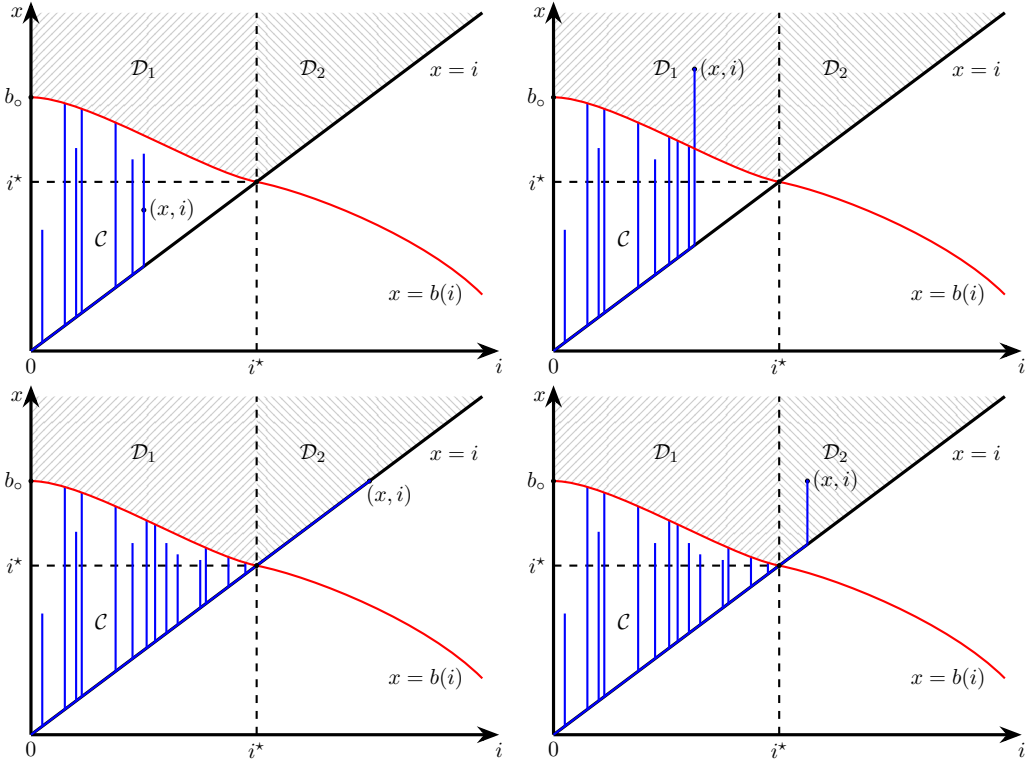

Notice then that the component $\mathcal{D}_1$ of the action region corresponds to the Scenario (a) in Section \ref{sec:processsetting}. 
In particular, if the surplus process $X^{x,D}$ starts from a value $x$ and has a historical worst performance $i$, such that $(x,i) \in \mathcal{D}_1$, 
then the company should pay a lump-sum dividend at time $0$ of a relatively small size 
$$
\Delta\D_0 = x-b(i) \in [0, x-i) 
\quad \text{and} \quad 
(X^{x,\D}_{0}, I^{i,\D}_{0}) = (b(i), i) \in \mathcal{M}^0 \setminus \partial \mathcal{M}^0,
$$
since $X^{x,\D}_{0} = b(i) > i$.  
In other words, the process $(X^{x,\D}, I^{i,\D})$ jumps downwards in the $x$-axis from $(x,i)$ to $(b(i),i)$ at time $0$, and then continues as for $v(b(i),i)$, which by continuity is equal to the corresponding expression \eqref{eq:vonC} (see top right panel in Figure \ref{Fig2}). 

In view of the above, we can use Definition \ref{int:D} for the initial jump to define 
\begin{align} \label{eq:vonD1}
\begin{split}
v(x,i) 
&:= \int_0^{\Delta \D_0} e^{ - q i} \ \d u + v(b(i),i) \\ 
&= e^{-q i} \big(x-b(i) \big) + v(b(i),i) 
= e^{-q i} \Big(x - b(i) + \frac{\mu}{\rho}\Big), \quad (x,i)\in \mathcal{D}_1. 
\end{split}
\end{align}
where the last equality can be easily derived from \eqref{eq:vonC} and the definitions of $\alpha$ and $\beta$,~which~yields
$$
v(b(i), i)
= \frac{\alpha+\beta}{\alpha \beta} \,  e^{-q i} 
= \frac{\mu}{\rho} \,  e^{-q i}. 
$$

We finally focus on the component $\mathcal{D}_2$ of the action region, which, on the contrary, corresponds to the Scenarios (b)--(c) in Section \ref{sec:processsetting}. 
In particular, if the surplus process $X^{x,D}$ starts from a value $x$ and has a historical worst performance $i$, such that $(x,i) \in \mathcal{D}_2$, 
then the company should pay a lump-sum dividend at time $0$ of a relatively large size 
$$
\Delta\D_0 = x-i^\star \in (x-i, x) 
\quad \text{and} \quad 
(X^{x,\D}_{0}, I^{i,\D}_{0}) = (b(i^\star), i^\star) \in \partial \mathcal{M}^0,
$$
since $X^{x,\D}_{0} = b(i^\star) < i$. 
In other words, the process $(X^{x,\D}, I^{i,\D})$ moves either downwards in a 45$^\circ$-angle along the diagonal $\partial \mathcal{M}$ from $(i,i)$ to $(i^\star,i^\star)$ if $x=i$ (see bottom left panel in Figure \ref{Fig2}), 
or downwards in a `hockey-stick' direction, firstly in the $x$-axis from $(x,i)$ to $(i,i)$ and subsequently in a 45$^\circ$-angle along the diagonal $\partial \mathcal{M}$ from $(i,i)$ to $(i^\star,i^\star)$ if $x>i$ (see bottom right panel in Figure \ref{Fig2}). 
In either case, it then continues as for $v(i^\star,i^\star)$, which by continuity is equal to the corresponding expression in \eqref{eq:vonC}.

In view of the above, we can use Definition \ref{int:D} for the initial jump to define 
\begin{align} \label{eq:vonD2}
\begin{split}
v(x,i) 
&:= \int_0^{x-i} e^{ - q i} \ \d u + \int_{x-i}^{\Delta \D_0} e^{ - q (x-u)} \ \d u + v(b(i^\star),i^\star) \\
&= e^{-q i}(x-i) + \frac{1}{q} \big( e^{-q i^\star} - e^{-q i} \big) + \frac{\mu}{\rho} \,  e^{-q i^\star} \\
&= e^{-q i} \Big(x - i - \frac{1}{q}\Big) + \Big( \frac{1}{q} + \frac{\mu}{\rho} \Big) e^{-q i^\star}, \quad (x,i)\in \mathcal{D}_2,
\end{split}
\end{align}
where the second equality can be easily derived from \eqref{eq:vonC}, the property $b(i^\star)=i^\star$ and the definitions of $\alpha$ and $\beta$, which imply
$$
v(b(i^\star), i^\star)
= v(i^\star,i^\star) 
= \frac{\alpha+\beta}{\alpha \beta} \,  e^{-q i^\star} 
= \frac{\mu}{\rho} \,  e^{-q i^\star}. 
$$

Collecting together the expressions \eqref{eq:vonC}, \eqref{eq:vonD1} and \eqref{eq:vonD2} derived above for our candidate value function $v$ in the waiting region $\mathcal{C}$ and action region $\mathcal{D}_1 \cup \mathcal{D}_2$, we have  
\begin{equation} \label{eq:v}
v(x,i) = \begin{cases} 
\frac{1}{\alpha-\beta} \, e^{-q i} \big(\frac{\alpha}{\beta}e^{\beta(x-b(i))} - \frac{\beta}{\alpha}e^{\alpha(x-b(i))}\big), &\quad (x,i) \in \mathcal{C}, \\
e^{-q i} \big(x - b(i) + \frac{\mu}{\rho}\big), &\quad (x,i)\in \mathcal{D}_1, \\
e^{-q i} \big(x - i - \frac{1}{q}\big) + \big( \frac{1}{q} + \frac{\mu}{\rho} \big) e^{-q i^\star}, &\quad (x,i)\in \mathcal{D}_2.
\end{cases}
\end{equation}

\begin{theorem} \label{thm:vsolvesHJB}
The couple $(v,b):(0,\infty)^2 \times (0,\infty) \to [0,\infty) \times (-\infty,b_\circ)$ defined through \eqref{eq:v} and the unique solution to the ODE \eqref{eq:ODEb}--\eqref{flow} in Theorem \ref{thm:ODEb} 
solves the free-boundary problem \eqref{eq:FBP1}-\eqref{eq:FBP6}. 
\end{theorem}
\begin{proof}
By construction, it is enough to check that:
\begin{equation*}
(i) \quad v_x(x,i)\geq e^{-qi} \text{ on } \mathcal{C} 
\quad \text{and} \quad 
(ii) \quad [(\mathcal{L}_X - \rho)v](x,i) \leq 0 \text{ on } \mathcal{D}_1 \cup \mathcal{D}_2.
\end{equation*}
We prove these inequalities in the following three steps. 

\vspace{1mm}
{\it Step 1: Proof of} $(i)$. 
Direct computations reveal that for any $(x,i) \in \mathcal{C}$, we have the equivalence
$$
v_x(x,i)\geq e^{-qi} \quad 
\Leftrightarrow \quad 
\alpha e^{\beta(x-b(i))} - \beta e^{\alpha(x-b(i))} \leq \alpha-\beta.
$$
We then define $F:\overline{\mathcal{C}} \to \R$ by
$$
F(x,i):= \alpha e^{\beta(x-b(i))} - \beta e^{\alpha(x-b(i))}
$$
and observe that 
$F(b(i),i) = \alpha - \beta$ 
and  $$
F_x(x,i) = \alpha\beta (e^{\beta(x-b(i))} - e^{\alpha(x-b(i))}) >0, \quad (x,i) \in \mathcal{C}, 
$$
due to the inequalities $\beta>0>\alpha$. 
Hence, we conclude that $F(x,i) < \alpha - \beta$ on $\mathcal{C}$, which implies the required inequality.

\vspace{1mm}
{\it Step 2: Proof of} $(ii)$ {\it on $\mathcal{D}_1$}.
Recall the expression \eqref{eq:vonD1} of $v$ on $\mathcal{D}_1$, which yields that $v_x(x,i)=e^{-qi}$ and $v_{xx}(x,i)=0$. 
Hence,
\begin{align} \label{eq:LonD1}
\begin{split}
[(\mathcal{L}_X - \rho)v](x,i) 
&= \mu e^{-qi} - \rho e^{-qi} \Big(x - b(i) + \frac{\mu}{\rho}\Big) \\
&= - \rho e^{-qi} (x-b(i)) \leq 0, \quad (x,i) \in \mathcal{D}_1,
\end{split}
\end{align}
where we have used that $x \geq b(i)$ on $\mathcal{D}_1$.

\vspace{1mm}
{\it Step 3: Proof of} $(ii)$ {\it on $\mathcal{D}_2$}. 
Recall now the expression \eqref{eq:vonD2} of $v$ on $\mathcal{D}_2$, which also yields that $v_x(x,i)=e^{-qi}$ and $v_{xx}(x,i)=0$. 
Hence, 
\begin{align} \label{eq:LonD2}
\begin{split}
[(\mathcal{L}_X - \rho)v](x,i) 
&= \mu e^{-qi} - \rho\Big[e^{-q i} \Big(x - i - \frac{1}{q}\Big) + \Big( \frac{1}{q} + \frac{\mu}{\rho} \Big) e^{-q i^\star} \Big] \\
&= -\Big(\frac{\rho}{q}+\mu\Big)\big(e^{-q i^\star} - e^{-q i}\big) -\rho e^{-q i}(x-i)\leq 0, \quad (x,i) \in \mathcal{D}_2,
\end{split}
\end{align}
where we have used that $x \geq i$ and that $i \geq i^\star$ on $\mathcal{D}_2$.
\end{proof}

\subsubsection{Construction of the candidate optimal control}
\label{sec:control}

We are now ready to construct a candidate optimal control $\D^\star$. 
To that end, recalling the structure of the action region $\mathcal{D}_1 \cup \mathcal{D}_2$ in \eqref{eq:stopping}, we define the potentially optimal initial jump of the control $\D^\star$ by 
\begin{equation} \label{OC-time0}
\D^{\star}_{0^-}:=0 
\quad \text{and} \quad 
\D^{\star}_0:=\big(x-b(i) \big) \mathds{1}_{\{(x,i) \in \mathcal{D}_1\}} + \big( x-b(i^{\star}) \big) \mathds{1}_{\{(x,i) \in \mathcal{D}_2\}}.
\end{equation}

Then, for $k\geq 1$, we define $\P_{x,i}$-a.s.\ the process $\D^\star$ by 
\begin{align} \label{eq:candidateOC}
\begin{split}
& \D^{\star}_t := \sup_{s \in [0,t]} \big( x + \mu s + \eta W_s  - b(i)\big)^+, \quad t \in [0,\lambda_{1}], \\
& \D^{\star}_t :=\D^{\star}_{\lambda_k}, \quad t \in [\lambda_{k},\gamma_k] \\
& \D^{\star}_t := \D^{\star}_{\gamma_{k}}  + 
\sup_{s \in [\gamma_k,t]} \big( X^{\D^{\star}}_{\gamma_k} + \mu \big(s-\gamma_k \big) + \eta \big(W_s - W_{\gamma_{k}} \big) - b(I^{\D^{\star}}_{\gamma_k})\big)^+, \quad t \in [\gamma_{k},\lambda_{k+1} ], 
\end{split}
\end{align} 
where the controlled process $X^{\D^{\star}}_t$ is given by \eqref{XI}, and where we have set $\gamma_0=0$ and, for $k\geq 1$,
\begin{align*}
\lambda_k &:= \inf\{t>\gamma_{k-1}:\, X^{D^{\star}}_t = I^{D^{\star}}_{\gamma_{k-1}}\}, \\
\gamma_k &:= \inf\{t>\lambda_k:\, X^{D^{\star}}_t = b(i)\},
\end{align*}
with the usual convention $\inf\emptyset = \infty$.

\begin{proposition} 
\label{prop:OC}
For any $(x,i)\in \mathcal{M}$, recall the waiting region $\mathcal{C}$ defined by \eqref{CD}, the set $\mathcal{A}(x,i)$ of admissible controls in \eqref{set:A} and the control $\D^\star$ defined by \eqref{OC-time0}--\eqref{eq:candidateOC}. 
Then, we have that $\D^\star \in \mathcal{A}(x,i)$. 
Furthermore, we have on $\{t < \tau^{\D^\star}_0\}$ that 
\begin{equation}  \label{eq:SkcondOC-1}
\big(X^{\D^\star}_t, I^{\D^\star}_t \big) \in \mathcal{C}, \quad \P_{x,i} \otimes \d t-\text{a.e.}
\end{equation}
and, $\P_{x,i}$-a.s., that
\begin{equation} \label{eq:SkcondOC-2}
\int_{[0,\tau^{\D^\star}_0]} \mathds{1}_{\{X^{\D^{\star}}_{u^-} < b(I^{\D^{\star}}_{u})\}} \d \D^{\star}_u 
= \int_0^{\Delta \D^{\star}_t} \hspace{-2mm} \mathds{1}_{\{X^{\D^\star}_{t^-} -\zeta < b(I^{\D^{\star}}_{t})\}} \d\zeta = 0 .
\end{equation}
\end{proposition}

\begin{proof}
We prove the different parts separately. 

\vspace{1mm}
{\it Admissibility.} 
It follows by construction that $\D^{\star}_{0^-}=0$, $\D^{\star}$ is $\mathbb{F}$-adapted, non-decreasing, and right-continuous (with the only potential discontinuity being at initial time) .

Next, we show that $X^{x,\D^\star}_{t^-} - \Delta \D^\star_t \geq 0$ holds true for all $t\in[0,\tau_0^{\D^\star}]$, $\P$-a.s. 
By combining the fact that $\D^\star$ can have a jump potentially only at time $0$ (by construction) and its amplitude is given by $D^\star_0$ in \eqref{OC-time0}, together with the fact that $b(i) > 0$ for all $i\in[0,i^\star]$ thanks to Lemma \ref{lem:i^*}, it is straightforward to conclude the desired inequality from the observation that
$$
X^{x,\D^\star}_{t^-} - \Delta \D^\star_t 
= \begin{cases}
(x-b(i)) \wedge (x-b(i^\star)) >0, &\text{for } t=0 \text{ and } (x,i)\in \mathcal{D}, \\
X^{x,\D^\star}_{t^-} \geq 0, &\text{otherwise}. 
\end{cases}
$$
Finally, we notice due to the non-negativity of $I^{\D^\star}_t$ for all $t \in [0,\tau^{\D^\star}_0]$ that 
\begin{align*}
\E_{x,i}\bigg[\int_0^{\tau_0^{\D^\star}} e^{-\rho t - q I^{\D^\star}_t} \, {\diamond}\, \d \D^\star_t \bigg] 
&\leq \E_{x}\bigg[\int_0^{\tau_0^{\D^\star}} e^{-\rho t} \, \d \D^\star_t \bigg] \\
&\leq \sup_{\D \in \tilde{\mathcal{A}}(x)} \E_{x} \bigg[\int_0^{\tau_0^{\D}} e^{-\rho t} \, \d \D_t \bigg] 
= V_0(x) < \infty ,
\end{align*}
where we used also the fact that $\D^\star \in \mathcal{A}(x,i) \subseteq \tilde{\mathcal{A}}(x)$ and the definitions \eqref{eq:value-DF0} and \eqref{set:Ax} of $V_0$ and $\tilde{\mathcal{A}}(x)$, respectively. 

\vspace{1mm}
{\it Properties \eqref{eq:SkcondOC-1}--\eqref{eq:SkcondOC-2}.} 
It is readily seen by the definition of $\D^{\star}$ that \eqref{eq:SkcondOC-1} holds true. 

It thus remains to show that the two sums in \eqref{eq:SkcondOC-2} are zero.
It follows by the construction in \eqref{OC-time0}--\eqref{eq:candidateOC} that
\begin{align*}
&\D_t(\omega) \text{ is constant }  
\forall \; 0 \leq t \leq \lambda_{1}(\omega) \wedge \tau^{\D^\star}_0(\omega) \text{ such that } X^{\D^\star}_{t-}(\omega)<b(i), \\
&\D_t(\omega) \text{ is trivially  constant } \forall \; \lambda_{k}(\omega) \leq t \leq \gamma_k(\omega) \wedge \tau^{\D^\star}_0(\omega), \\
&\D_t(\omega) \text{ is constant } 
\forall \; \gamma_{k}(\omega) \leq t \leq \lambda_{k+1}(\omega) \wedge \tau^{\D^\star}_0(\omega) \text{ such that } X^{\D^\star}_{t-}(\omega)<b(I^{\D^{\star}}_t(\omega)) = b(I^{\D^{\star}}_{\gamma_{k}}(\omega)),
\end{align*}
hence the first sum in \eqref{eq:SkcondOC-2} is clearly zero. 
For the second sum, notice that $\D^\star$ can have a jump potentially only at time $0$ and its amplitude is given by $\D^\star_0$ in \eqref{OC-time0}, so that 
\begin{align*}
\int_0^{\Delta \D^{\star}_t} \hspace{-2mm} \mathds{1}_{\{X^{\D^\star}_{t^-} -\zeta < b(I^{\D^{\star}}_{t})\}} \d\zeta 
&= \int_{0}^{\Delta \D^{\star}_0} \hspace{-2mm} \mathds{1}_{\{\zeta > x - b(I^{\D^{\star}}_{0})\}} \d\zeta \\
&= \mathds{1}_{\{(x,i) \in \mathcal{D}_1\}} \int_{0}^{x-b(i)} \hspace{-4mm} \mathds{1}_{\{\zeta > x-  b(i) \}} \d\zeta 
+
\mathds{1}_{\{(x,i) \in \mathcal{D}_2\}} \int_{0}^{x-b(i^\star)} \hspace{-4mm} \mathds{1}_{\{\zeta > x-  b(i^\star) \}} \d\zeta
= 0.
\end{align*}
Hence, we conclude that the second sum is also zero, and complete the proof.
\end{proof}

Collecting Theorem \ref{thm:vsolvesHJB} and the previous construction, in light of Theorem \ref{thm:verification} we have the following final theorem.

\begin{theorem} \label{thm:final}
The function $v$ defined by \eqref{eq:v} identifies with the value function $V$ of the singular control problem \eqref{eq:value-div}, i.e.~$V\equiv v$, and the control $\D^\star$ defined by \eqref{OC-time0}--\eqref{eq:candidateOC} is optimal.
\end{theorem}
\begin{proof}
In order to prove this result, we rely on the application of the Theorem \ref{thm:verification}.

To that end, since $v$ solves the free-boundary problem \eqref{eq:FBP1}--\eqref{eq:FBP6} thanks to Theorem \ref{thm:vsolvesHJB}, then it immediately verifies the system \eqref{eq:HJB}--\eqref{eq:bdconds2}. 

In order to see the validity of \eqref{eq:integrability}, notice that 
$$
\E_{x,i}\bigg[\sup_{t \in [0,\tau^{\D}_a]} e^{-\int_0^{t} r(X^{\D}_s,I^{\D}_s) \, \d s} \, \big|v\big(X^{\D}_{t},I^{\D}_{t}\big)\big|\bigg] 
\leq \E_{x,i} \Big[\sup_{t \geq 0}  e^{-rt} \, C \big(1 + \big|X^{0}_{t} \big| \big) \Big] 
< + \infty , 
$$
where we used the bound $|v(x,i)| \leq C(1+x)$ from \eqref{eq:v} and the property that $X^\D_t \leq X^0_t$, for all $t\geq 0$ by definition \eqref{XI}. 
Furthermore, similar estimates lead to the verification of \eqref{eq:transversality} as well. 

Finally, given that the process $\D^\star$ satisfies the properties in Proposition \ref{prop:OC}, it is immediate that \eqref{eq:Skh1} and \eqref{eq:Skh2} hold true, which then completes the proof.
\end{proof}

As depicted in Figure \ref{Fig2}, the optimal dividend policy $\D^\star$ defined by \eqref{OC-time0}--\eqref{eq:candidateOC} shows a behaviour which is remarkably different from that of the classical dividend problem with $q=0$, where the optimal policy is triggered by the \emph{constant barrier} $b_o$ and pays just enough dividends to prevent the surplus from entering the region $\{x \in \R :\, x>b_o\}$ (see Section \ref{sec:q=0} for details).
According to \eqref{OC-time0}--\eqref{eq:candidateOC}, when the decision maker has a time-preference rate $q>0$, the optimal dividend payout is instead triggered by the \emph{infimum-dependent} barrier strategy $b(i)$. This critical level is dynamically updated according to the current level of the surplus process' running infimum, whose value in turn depends on the policy in action.
In particular, on the one hand, when the surplus process starts in region $\mathcal{D}_1$, it is optimal to pay a lump-sum dividend of a relatively small size $x-b(i)$ at time $0$ and then pay continuously in a manner such that the process $X^{x,\D}$ is reflected downwards on $b(I^{i,\D})$ according to a Skorokhod reflection (cf.\ \eqref{eq:candidateOC}). 
On the other hand, if $x\in \mathcal{D}_2$, it is optimal to pay a lump-sum dividend of a relatively large size $x-b(i^*)$ at time $0$, thus $(X^{x,\D}, I^{i,\D})$ enters 
the region $\mathcal{C} \cup \mathcal{D}_1$ and proceeds as above.

Nevertheless, it turns out that the solution to the problem with multiplicative time-preference impact (i.e.\ with $q>0$) converges to the one of the classical De Finetti's optimal dividend problem (see Section \ref{sec:q=0}) as $q\downarrow 0$. This is proved in the next proposition, 
where we denote by $b(\cdot)=b(\cdot;q)$ the solution the ODE \eqref{eq:ODEb} and by $V(\cdot) = V(\cdot;q)$ the value functions \eqref{eq:valuefctmuneg} and \eqref{eq:v} to stress their dependence on the parameter $q$.

\begin{proposition}
Recall the definition \eqref{eq:bo} of $b_\circ$ and the value function expressions \eqref{eq:V0muneg} and \eqref{eq:sol-DF0} of $V_0(\cdot)$ when $q=0$. 
For any $q>0$, recall also the ODE \eqref{eq:ODEb} satisfied by the function $b(\cdot;q)$ and the value function expressions \eqref{eq:valuefctmuneg} and \eqref{eq:v} of $V(\cdot,\cdot;q)$. 
Then, we have 
\begin{align} \label{eq:stabilityb}
\lim_{q \downarrow 0}b(i;q) &= b_o, \quad \forall \; i\geq 0, \\ \label{eq:stabilityV} 
\lim_{q \downarrow 0}V(x,i;q) &= V_0(x), \quad \forall \; (x,i) \in \mathcal{M}^0.
\end{align}
\end{proposition}
\begin{proof}
The limit in \eqref{eq:stabilityb} follows by classical results on continuous dependence of unique solutions to ODEs with respect to a parameter; see, e.g., \cite[Chapter~V, Theorem~2.1]{Hartman}.

In the case of $\mu\leq0$, the result in \eqref{eq:stabilityV} is readily obtained by taking limits as $q \downarrow 0$ in \eqref{eq:valuefctmuneg} to obtain \eqref{eq:V0muneg}.

In the case of $\mu>0$, when taking limits in \eqref{eq:v} as $q \downarrow 0$, one first notices that $\mathcal{D}_1 \cup \mathcal{D}_2$ converges to the region $\{x\in\R:\, x \geq b_o\}$ due to \eqref{eq:stabilityb}. 
Also, the critical value $i^{\star}=i^{\star}(q)$ obtained in Lemma \ref{lem:i^*} satisfies $\lim_{q\downarrow 0} i^{\star}(q) = b_\circ$ again due to  \eqref{eq:stabilityb}. 
Hence, it is a matter of standard limit arguments to show that both of the expressions of $V$ in $\mathcal{D}_1$ and $\mathcal{D}_2$ (cf.\ \eqref{eq:v} and Theorem \ref{thm:final}) converge to $x-b_o + \frac{\mu}{\rho}$, which is indeed the expression of $V_0$ in the action region $\{x\in\R:\, x \geq b_o\}$ in \eqref{eq:sol-DF0}. 
Finally, the convergence of $V$ to $V_0$ as $q \downarrow 0$ within the region $\mathcal{C}$ is obvious, again due to \eqref{eq:stabilityb}.
\end{proof}

\begin{remark}
\label{rem:extension}
Notice that, there is a plethora of applications that would fit the general framework of Section \ref{sec:problem}. 
We present below another variant of the classical dividend problem, inspired by \cite{Schmidli}. 
According to \cite{Schmidli}, the company pays dividends to shareholders, discounts costs and profits at a deterministic and constant rate $\rho$, and even though the company is not forced into bankruptcy when its surplus becomes negative, it incurs running penalty payments over time, depending on the current level of the surplus. These penalty payments prevent losses from growing without bound and can be interpreted as a preference measure or as costs associated with negative capital.

A possible extension of this formulation would be the introduction of an additional, transitional cost for diminishing the past worst performance $I^D$ of the company. This cost would penalise the company even when the surplus process is positive, by being activated whenever the past worst performance $I^D$ of the company reaches a new minimum. 
This would represent an additional precautionary, preference measure, with  penalty payments that would prevent the surplus process from entering the negative region.
In this formulation, the company aims to solve
\begin{equation*}
\sup_{\D \in \mathcal{D}(x,i)}\E_{x,i}\bigg[\int_{0}^{\infty} e^{-\rho t} \Big(\d \D_t -  \phi(X^{\D}_t) \d t + h(I^{\D}_t) \  {\scriptstyle{\Box}}\ \d I^{\D}_t \Big) \bigg],
\end{equation*}
where 
\begin{eqnarray*}
& \hspace{0.2cm}\mathcal{D}(x,i) := 
\bigg\{ \D \in \mathcal{S}:\, X^{x,\D}_{t} \in \mathcal{O},\,\, \displaystyle \E_{x,i}\bigg[\int_0^{\infty} e^{-\rho t} \Big(\big|\phi\big(X^{\D}_t)\big| \, \d t +  \d \D_t + \big|h\big(I^{\D}_t\big)\big| \, {\scriptstyle{\Box}}\, \d I^{\D}_t\Big)\bigg] < \infty \bigg\}, \nonumber
\end{eqnarray*}
$\phi$ is a non-increasing function that converges to zero as the argument diverges to infinity (as in \cite{Schmidli}), and $h$ is a suitable non-negative Borel-measurable cost function.
This variant of the dividend problem fits in the general formulation of the stochastic control problem in Section \ref{problem} and also involves the integral with respect to the infimum process introduced in \eqref{defintegral-I}.
This novel problem can also be studied using the methodology presented in this paper; its rigorous treatment is left for future research.
\end{remark}


\section{Integrals for state-spaces involving the running supremum process}
\label{sec:sup}

In this section, we aim at presenting the novel framework of this paper in the setup where the underlying two-dimensional process is given by a diffusion and its running supremum. 
The aim is to show how the integrals with respect to the controlled version of the diffusion and with respect to the controlled running maximum process of the diffusion can be obtained via the Definitions \ref{int:D}--\ref{int:I}. 
These are the novel integral operators that are consistent with the Hamilton-Jacobi-Bellman equation and could be used in the study of two-dimensional singular control problems involving the controlled diffusion and its controlled running maximum process (in the spirit of Section \ref{problem}).

We consider on $(\Omega,\mathcal{F},\P)$ the process 
\begin{equation}
\label{defY}
Y_t^{y,D} := -X_t^{x,D}, \quad \text{for} \quad y:=-x \quad \text{and all} \quad  t \geq 0, 
\end{equation}
satisfying in view of \eqref{stateX} the SDE 
\begin{equation}
\label{stateY}
\d Y_t^{y,D} = \tilde b(Y_t^{y,D}) \d t + \tilde \ss (Y_t^{y,D}) \d B_t + \d\D_t, \quad Y_{0^-}^{y,D}  = y \in \mathcal{I} := (\underline y, \overline y),
\end{equation}
where $- \infty \leq \underline y := - \overline x < \overline y := - \underline x \leq + \infty$ and $\tilde b(y) := - b(-y)$ and $\tilde \ss(y) := - \ss(-y)$ which clearly satisfy Assumption \ref{A1}, so that for any $y \in \mathcal{I}$ and any $\D \in \mathcal{S}$, there exists a unique strong solution $Y^{y,\D}$ to \eqref{stateY}.
Note that, by exerting control via the process $\D$, the decision maker increases the level of $Y^{y,\D}$.

Using now the definition \eqref{defY} of $Y^{y,\D}$ in the definition \eqref{eq:infim} of the running infimum $I^{i,\D}$ of $X^{x,\D}$, we notice that 
\begin{equation} \label{sup=inf}
I^{i,\D}_t 
= i \wedge \inf_{0\leq u \leq t} (-Y^{y,\D}_u)
= - \big( -i \vee \sup_{0\leq u \leq t} Y^{y,\D}_u \big)
= - S^{s,\D}_t, 
\quad \text{for} \quad s:=-i \quad \text{and} \quad t \geq 0,
\end{equation}
where we define for any given $y, s \in \mathcal{I}$ such that $y \leq s$ and $\D \in \mathcal{S}$, the running supremum $S^{s,\D}$ of $Y^{y,\D}$ (also controlled by $\D$) by 
\begin{equation} \label{eq:sup}
S^{s,\D}_t 
:= s \vee \sup_{0\leq u \leq t} Y^{y,\D}_u, \quad \text{for} \quad t \geq 0.
\end{equation}

In light of the definition \eqref{defY} and the equation \eqref{sup=inf}, the two-dimensional controlled process $(Y^{y,\D}, S^{s,\D})$ therefore satisfies
\begin{equation} \label{eq:YS}
(Y^{y,\D}_t, S^{s,\D}_t) = (- X^{x,\D}_t, - I^{i,\D}_t)
\quad \forall\;  t\geq 0, \quad \P-\text{a.s.},
\end{equation}
and its associated state-space is consequently given by
\begin{equation}
\label{eq:setN}
{\mathcal{N}} := \{(y,s) \in \mathcal{I} \times \mathcal{I}:\, y \leq s\}.
\end{equation}
Notice that $S^{s,\D}$ is a non-decreasing process, which is increasing on the diagonal of its state-space 
$$
\partial \mathcal{N} 
:= \{(y,s) \in \mathcal{I} \times \mathcal{I} :\, y = s\} 
$$
and remains constant while the process $(Y^{y,\D}, S^{s,\D})$ is away from the diagonal, namely in 
$\mathcal{N} \setminus \partial \mathcal{N}$. 
Combining the above with Scenarios (a)--(c) in Section \ref{sec:processsetting} yields that the two-dimensional process $(Y^{y,\D}, S^{s,\D})$ can move 
{\it upwards in the $y$-axis} while maintaining the same coordinate value in the $s$-axis, 
or 
{\it upwards in a 45$^\circ$-angle direction along the diagonal $\partial \mathcal{N}$}, i.e.~in the north-east direction,  or 
{\it upwards in an `inverse hockey-stick' direction, firstly in the $y$-axis and subsequently in a 45$^\circ$-angle along the diagonal $\partial \mathcal{N}$}, as a combination of the former two movements.

It is therefore a direct consequence of using \eqref{eq:YS} in Definitions \ref{int:D}--\ref{int:I} to obtain the following
expressions for the integral operators that address the aforementioned movements of $(Y^{y,\D}, S^{s,\D})$ induced by a control process $\D$ and are consistent with the Hamilton-Jacobi-Bellman equation. 
For the stochastic process $(Y^{y,\D}, S^{s,\D})$ defined by \eqref{stateY} and \eqref{eq:sup} and its state-space $\mathcal{N}$ given by \eqref{eq:setN}, for any $T>0$ and generic functions $r, g: \mathcal{N} \to \R$ (such that the following quantities are well-defined), we have 
\begin{align} \label{defintegralY}
\begin{split}
&\int_0^{T} e^{-\int_0^t r(Y^{\D}_z, S^{\D}_z)\d z} \, g\big(Y^{\D}_t, S^{\D}_t\big) \ {\diamond}\ \d \D_t \\
& =\int_{0}^{T} e^{-\int_0^t r(Y^{\D}_z, S^{\D}_z)\d z} \, g\big(Y^{\D}_t, S^{\D}_t\big)\ \d \D^c_t \\
& \hspace{0.5cm} 
+ \sum_{t \leq T:\,\, \Delta \D_t \neq 0} e^{-\int_0^t r(Y^{\D}_z, S^{\D}_z)\d z} \int_0^{(S^{\D}_{t^-} - Y^{\D}_{t^-}) \wedge \Delta \D_t} g\big(Y^{\D}_{t^-}+u, S^{\D}_{t^-}\big)\ \d u \\
& \hspace{0.5cm} 
+ \sum_{t \leq T:\,\, \Delta \D_t \neq 0} e^{-\int_0^t r(X^{\D}_z, I^{\D}_z)\d z} \int_{S^{\D}_{t^-} - Y^{\D}_{t^-}}^{\Delta \D_t} \mathds{1}_{\{\Delta \D_t > S^{\D}_{t^-} - Y^{\D}_{t^-}\}} \, g\big(Y^{\D}_{t^-}+u, Y^{\D}_{t^-}+u\big)\ \d u. 
\end{split}
\end{align}     
and 
\begin{align}
\label{defintegral-S}
\begin{split}
& \int_0^{T} e^{-\int_0^t r(Y^{\D}_z, S^{\D}_z)\d z} \, g\big(Y^{\D}_t, S^{\D}_t\big) \ {\scriptstyle{\Box}}\ \d S^{\D}_t \\
&= \int_{0^-}^{T} e^{-\int_0^t r(Y^{\D}_z, S^{\D}_z)\d z} \, g\big(Y^{\D}_t, S^{\D}_t\big)\ \d S^{\D,c}_t \\
& \hspace{0.5cm} + \sum_{t \leq T:\,\, \Delta \D_t \neq 0} e^{-\int_0^t r(Y^{\D}_z, S^{\D}_z)\d z} 
\int_{S^{\D}_{t^-} - Y^{\D}_{t^-}}^{\Delta \D_t} \mathds{1}_{\{\Delta \D_t > S^{\D}_{t^-} - Y^{\D}_{t^-}\}} \, g\big(Y^{\D}_{t^-}+u, Y^{\D}_{t^-}+u\big)\ \d u,
\end{split}
\end{align}
where $S^{\D,c}$ is the continuous part of the non-decreasing process $S^{\D}$ with the induced (random) measure $\d S^{\D,c}_\cdot(\omega)$ having support on $\{t\geq0 : Y^{\D}_t(\omega) = S^{\D}_t(\omega)\}$, $\omega \in \Omega$.


\vspace{0.4cm}

\textbf{Acknowledgements.} Supported  by the Deutsche Forschungsgemeinschaft (DFG, German Research Foundation) - Project-ID 317210226 - SFB 1283. We thank the anonymous associate editor and referees for their valuable comments and suggestions on an earlier version of this paper.


\end{document}